\documentclass[a4paper,reqno,final]{amsart}

\usepackage{graphicx,enumerate,nicefrac,color,bm,cite}
\usepackage{amsmath, amssymb, amstext, amsfonts,amsthm}
\usepackage[foot]{amsaddr}
\usepackage{booktabs}
\usepackage{a4wide}
\usepackage{xcolor}
\usepackage{mathtools}
\usepackage{algorithm,algpseudocode,tabularx}

\makeatletter
\newcommand{\multiline}[1]{%
  \begin{tabularx}{\dimexpr\linewidth-\ALG@thistlm}[t]{@{}X@{}}
    #1
  \end{tabularx}
}
\makeatother

\newcommand{\spc}[1]{\mathbb{#1}}
\newcommand{\set}[1]{\mathcal{#1}}
\newcommand{\Norm}[1]{\left\|#1\right\|}
\newcommand{\norm}[1]{\left\|#1\right\|_{\X}}

\newcommand{\dnorm}[1]{\left\|#1\right\|_{\X^\prime}}
\newcommand{\Lnorm}[1]{\left\|#1\right\|_{\Ltwo}}
\newcommand{\Lpnorm}[2]{\left\|#1\right\|_{\mathrm{L}^{#2}(\Omega)}}

\newcommand{\X}{\spc{X}}
\newcommand{\Y}{\spc{Y}}
\newcommand{\R}{\spc{R}}

\newcommand{\uN}{u_N}
\newcommand{\uNn}[1][n]{\uN^{#1}}
\newcommand{\M}{\set{M}}

\newcommand{\G}{\Lambda_f}

\newcommand{\operator}[1]{\mathsf{#1}}

\newcommand{\E}{\operator{E}}
\newcommand{\F}{\mathfrak{F}}

\newcommand{\T}{\operator{T}}
\renewcommand{\d}{\operator{d}}
\newcommand{\ds}{\,\d s}
\newcommand{\dt}{\,\d t}
\newcommand{\dx}{\,\d\x}

\newcommand{\dprod}[2]{\left<#1,#2\right>}

\newcommand{\mat}[1]{\bm{\mathsf{#1}}}
\newcommand{\x}{\mat x}

\newcommand{\CC}{\rho}
\newcommand{\un}[1][n]{u^{#1}}
\newcommand{\vn}[1][n]{v^{#1}}

\newcommand{\Hone}{\mathrm{H}^1_0(\Omega)}
\newcommand{\Ltwo}{\mathrm{L}^2(\Omega)}
\newcommand{\Lphi}{L_{\varphi}}

\newcommand{\comp}{\lhook\joinrel\relbar\mspace{-11mu}\hookrightarrow
}
\newcommand{\wstar}{\xrightharpoonup{*}}
\newcommand{\blank}{\scalebox{0.75}{$\bullet$}}

\renewcommand{\P}{\omega}
\newcommand{\Pref}{\widetilde{\P}}

\newcommand{\A}{\mathsf{A}}
\newcommand{\RN}{\mathsf{R}_N}
\DeclareMathOperator{\Span}{span}
\DeclareMathOperator*{\esssup}{ess\,sup}

\newtheorem{theorem}{Theorem}[section]
\newtheorem{lemma}[theorem]{Lemma}
\newtheorem{proposition}[theorem]{Proposition} 
\newtheorem{corollary}[theorem]{Corollary}

\theoremstyle{definition}

\newtheorem{remark}[theorem]{Remark}

\title[Iterative energy reduction and variational adaptivity]{Iterative energy reduction Galerkin methods\\ and variational adaptivity}

\author{Pascal Heid \and Thomas P.~Wihler}
\address{Mathematics Institute, University of Bern, Sidlerstr. 5, CH-3012 Bern, Switzerland}
\thanks{TW acknowledges the financial support of the Swiss National Science Foundation (SNSF), Grant No.~$200021\underline{\phantom{a}}212868$.}

\date{}

\begin{document}

\begin{abstract}
Critical points of energy functionals, which are of broad interest, for instance, in physics and chemistry, in solid and quantum mechanics, in material science, or in general diffusion-reaction models arise as solutions to the associated Euler–Lagrange equations. While classical computational solution methods for such models typically focus solely on the underlying partial differential equations, we propose an approach that also incorporates the energy structure itself. Specifically, we examine (linearized) iterative Galerkin discretization schemes that ensure energy reduction at each step, and utilize the computable discrete residual to determine an appropriate stopping point. Additionally, we provide necessary conditions, which are applicable to a wide class of problems, that guarantee convergence to critical points of the PDE as the discrete spaces are enriched. Moreover, in the specific context of finite element discretizations, we present a very generally applicable adaptive mesh refinement strategy---the so-called \emph{variational adaptivity approach}---which, rather than using classical a posteriori estimates, is based on exploiting local energy reductions. The theoretical results are validated for several computational experiments in the context of nonlinear diffusion-reaction models, thereby demonstrating the effectiveness of the proposed scheme.
\end{abstract}

\keywords{Energy-based algorithms, iterative energy reduction schemes, variational adaptivity, iterative nonlinear solvers, adaptive finite element methods, diffusion-reaction models}

\subjclass[2020]{35A15, 35B38, 65J15, 65M50, 65N30.}

\maketitle

\section{Introduction}

Given a (real) reflexive Banach space $\X$ endowed with the norm $\norm{\blank}$, this work is concerned with the numerical approximation of critical points of an energy functional 
\[
\E:\,\X\to (-\infty,\infty). 
\]
Such points arise as solutions of the associated Euler--Lagrange problem, given in weak form by:
\begin{equation}\label{eq:min}
\text{find } u^\star\in\M(z)\text{ such that } \dprod{\E'(u^\star)}{v}=0\qquad\forall v\in\X.
\end{equation}
Here, for a given point $z\in\X$, we restrict our search to the energy level set
\begin{equation}\label{eq:M}
\M(z):=\{v\in\X:\,\E(v)\le\E(z)\}.
\end{equation}
Throughout this work, we assume that the energy functional $\E$ is G\^ateaux differentiable on $\M(z)$ in the following sense: There exists an open superset $\set{G}\supset\set{M}(z)$ such that, for any $u\in\set{G}$, there is a bounded linear functional $\E'(u)\in\X^\prime$, the G\^ateaux derivative of $\E$ at $u$, satisfying
\[
\dprod{\E'(u)}{v}=\lim_{t\to 0}\frac{\E(u+tv)-\E(u)}{t}\qquad\forall v\in\X,
\]
where $\X^\prime$ denotes the dual space of $\X$, equipped with the dual norm $\dnorm{\ell}=\sup_{\norm{v}=1}|\dprod{\ell}{v}|$, for $\ell\in\X^\prime$, and $\dprod{\blank}{\blank}$ signifies the dual product on $\X'\times \X$.

In continuous models of practical interest in various scientific disciplines, the weak formulation~\eqref{eq:min} often arises on an infinite-dimensional Banach space~$\X$, with $\E':\,\X\to\X^\prime$ representing, for instance, a (possibly nonlinear) partial differential operator. In such cases, the numerical approximation of solutions to~\eqref{eq:min} necessitates a suitable discretization in a finite-dimensional setting and, for nonlinear problems, an interplay with an appropriate (iterative) linearization scheme. This combination of Galerkin discretizations and iterative solvers has been formalized and analyzed within the so-called \emph{iterative linearized Galerkin (ILG)} methodology; see~\cite{CongreveWihler:17,HeidWihler2:19v1,HeidWihler:19v2,HeidWihler:20}. 

Related approaches, which adopt a holistic perspective on the numerical approximation of (linear and nonlinear) partial differential equations by accounting for the interaction of various computational components, have been studied extensively in the past two decades; in particular, for finite element discretizations, we refer to~\cite{ChaillouSuri:07,GarauMorinZuppa:11,El-AlaouiErnVohralik:11,ErnVohralik:13,Carstensen:2014,BernardiDakroubMansourSayah:15, GantnerHaberlPraetoriusStiftner:17,bengharbia:hal-01919067,haberl2021convergence, BECKER202218, BRINGMANN2025102, Brunneretal:2023, BalciDieningStorn:2023, HeidSuli:2022,HeidPraetoriusWihler:2021,HeHoScWi:25,SpicherWihler:25} for a non-exhaustive list.
Most of these works rely on the key insight that a computationally \emph{efficient} interplay between the components of a numerical solution process requires well-designed (a priori or a posteriori) estimates. Such bounds enable the identification and control of individual error sources in the overall approximation in a purposeful manner. Typically, they are directly related to the Euler--Lagrange formulation~\eqref{eq:min}, for example through residual-type bounds. However, the underlying energy structure---i.e., potentially valuable information inherent in the energy functional~$\E$ itself---is usually not explicitly exploited.

Recently, a series of works focusing on specific variational PDE problems in physics and quantum chemistry (with associated energy frameworks) has shown that leveraging the energy topology can be highly advantageous, particularly in the development of adaptive finite element algorithms; see~\cite{HeidStammWihler:19, HeidJCP, AHW:23, HHSW:25}, as well as the related contributions~\cite{Vohralik:2024,harnist2024robust}. In contrast to traditional adaptive methods, it turns out that the energy-based approach can yield very effective numerical results even \emph{without} classical a posteriori residual or error estimates—an especially compelling advantage in challenging nonlinear (variational) problems where such bounds are difficult or impossible to obtain.

The goal of this paper is to propose a new energy-based paradigm built on two key components: \emph{Locally,} the discrete (and computable) residual is used to guide the energy reduction process toward a local minimum within a given Galerkin space. \emph{Globally,} the evolution of the energy of the (final) approximations on each discrete space is tracked, particularly how this sequence of energy values decays as the discrete spaces are enriched, in order to identify when the numerical solution approaches a local minimum. We derive a general mathematical framework for energy-based iterative numerical approximation schemes for variational problems and identify a set of sufficient conditions that enable a rigorous convergence analysis. Still on an abstract level, we further discuss the concept of \emph{variational adaptivity}, which seeks to exploit the energy structure of variational problems in the design of practical, energy-based adaptive algorithmic procedures; this idea is inspired by the recent works~\cite{HeidStammWihler:19, HeidJCP, AHW:23} on finite element approximations for specific semilinear models and eigenvalue problems  (see also~\cite{HoustonWihler:16} for an earlier variable-order approach), and will now be formulated within the general setting of the present paper. To illustrate the proposed methodology, we consider a class of nonlinear second-order elliptic boundary value problems in divergence form, and present corresponding computations within an adaptive finite element setting.

\subsubsection*{Outline}
In \S\ref{sec:IER}, we introduce a unified iteration scheme that guarantees energy reduction at each step. We then propose a procedure (Algorithm~\ref{alg:AILFEM}) that combines this general nonlinear solver with abstract Galerkin discretizations in an intertwined manner, and establish its convergence (Theorem~\ref{thm:convergence}). In \S\ref{sec:VA}, we discuss the concept of \emph{variational adaptivity}, aiming to transfer the abstract iterative linearized Galerkin energy-reduction methodology to a more practical computational setting. This is achieved by presenting an energy-based mesh-refinement strategy in the context of finite element discretizations (Algorithm~\ref{alg:refen}). In \S\ref{sec:drmodels}, we apply our abstract theory to nonlinear diffusion-reaction models. Finally, we present numerical experiments in \S\ref{sec:NE}, and draw some conclusions in~\S\ref{sec:concl}.

\section{Iterative energy reduction} \label{sec:IER}

For a given starting point $\un[0]\in\X$, which will take the role of $z$ in~\eqref{eq:M} above, we propose the following scheme for the iterative solution of~\eqref{eq:min}:
\begin{equation}\label{eq:iteration}
\dprod{\A[\un](\un[n+1]-\un[n])}{v}=-\dprod{\E'(\un)}{v}\qquad\forall v\in\X,\quad n\ge 0.
\end{equation}
Here, for each $n\ge 0$, we employ an \emph{invertible linear operator} $\A[\un]:\,\X\to\X^\prime$ that acts as a \emph{local linearization} of the G\^ateaux derivative $\E'$ at $\un$; in particular, for given $\un$, we note that the iteration~\eqref{eq:iteration} is a \emph{linear solve} for $\un[n+1]$ in each step $n\ge 0$. More generally, we assume that the following property, which will guarantee the well-posedness of the iteration~\eqref{eq:iteration}, holds true.
\begin{enumerate}[({A}1)] 

\item \emph{Invertibility:} For any given $w\in\M(\un[0])$, the mapping $\A[w]:\,\X\to\X^\prime$ is linear and invertible.

\end{enumerate}
In addition, we impose the following positivity condition:
\begin{enumerate}[({A}2)] 

\item \emph{Positivity:} For any $w\in\M(\un[0])$, the operator $\A[w]$ is positive; i.e. it holds that $\dprod{\A[w]v}{v}> 0$ for all $v\in\X$ with $v\neq0$. This condition further implies that the operator $\A[w]$ is continuous, or, equivalently, bounded; we refer to~\cite[Prop.~26.4]{Zeidler:90}. 

\end{enumerate}

In the sequel, we will formulate some sufficient conditions that lead the sequence $\{\un\}_{n\ge 0}$ generated by the iteration~\eqref{eq:iteration} to remain within the set $\M(\un[0])$ for all $n\ge 0$, and that yield convergence to a solution $u^\star$ of~\eqref{eq:min} in an appropriate sense.

\subsection{Linearized energy reduction}

The iteration~\eqref{eq:iteration} can be written in terms of an update operator
\[
\un\mapsto \un[n+1]:=\T(\un),\qquad n\ge 0,
\]
where, for given $w\in\M(\un[0])$, we define $\T(w)$ by the weak formulation
\begin{equation}\label{eq:T}
\dprod{\A[w](\T(w)-w)}{v}=-\dprod{\E'(w)}{v}\qquad\forall v\in\X.
\end{equation}
By assumption~(A1) above, we notice that the operator $\T$ is well-defined on~$\M(\un[0])$. The following property will be crucial in our analysis below.

\begin{enumerate}[(T)]

\item There exists $\gamma>0$ such that, for any $w\in\M(\un[0])$, the update operator $\T$ satisfies the bound
\begin{subequations}\label{eq:Tprop}
\begin{align} 
\E(w)-\E(\T(w))
&\ge -\gamma\dprod{\E'(w)}{\T(w)-w},
\intertext{or equivalently,}
\E(w)-\E(\T(w))
&\ge \gamma\dprod{\A[w](\T(w)-w)}{\T(w)-w}.\label{eq:Tpropb}
\end{align}
\end{subequations}
We note that the above inequalities are closely related to the celebrated Armijo–Goldstein condition in optimization; see, e.g., \cite[§2]{Armijo1966}.
\end{enumerate}

\begin{remark}\label{rem:Tequiv}
Property {\rm(T)} holds true if and only if there exists $\gamma>0$ such that
\begin{equation}\label{eq:Tequiv}
\int_0^1\dprod{\E'(s\delta_w+w)-\E'(w)}{\delta_w}\,\d s
\le (1-\gamma)\dprod{\A[w]\delta_w}{\delta_w}\qquad\forall w\in\X,
\end{equation}
where $\delta_w=\T(w)-w$. Indeed, by the fundamental theorem of calculus, for any $w\in\X$, we have
\begin{align*}
\E(w)-\E(\T(w))
&=-\int_0^1\frac{\d}{\d s}\E(s\delta_w+w) \, \d s 
= - \int_0^{1} \dprod{\E'(s\delta_w)}{\delta_w} \, \d s\\
&=-\dprod{\E'(w)}{\delta_w}-\int_0^1\dprod{\E'(s\delta_w+w)-\E'(w)}{\delta_w}\,\d s.
\end{align*}
Using~\eqref{eq:T}, reveals that property (T) is equivalent to~\eqref{eq:Tequiv}.
\end{remark}

We now investigate the convergence properties of the energy functional $\E$ in the context of the sequence generated by the iterative scheme~\eqref{eq:iteration}.

\begin{proposition}\label{prop:T}
If the update operator $\T$ defined in~\eqref{eq:T} satisfies condition {\rm(T)}, then, for any $w \in \M(\un[0])$, we have that $\E(\T(w)) \leq \E(w)$, and, in turn $\T(\M(\un[0]))\subset\M(\un[0])$. 
\end{proposition}

\begin{proof}
For any $w\in\M(\un[0])$, we let $\delta_w:=\T(w)-w$. Then, invoking~\eqref{eq:Tpropb} and employing the positivity condition~(A2), yields
\[
\E(w)-\E(\T(w))
\ge \gamma\dprod{\A[w]\delta_w}{\delta_w}\ge 0,
\]
whence we deduce that $\E(\T(w))\le\E(w)$.
\end{proof}

For the purpose of proving convergence of the iteration~\eqref{eq:iteration}, we require a further assumption.

\begin{enumerate}[({A}1)] 
\setcounter{enumi}{2}
\item \emph{Weak definiteness:} For any sequence $\{v^n\}_n\subset\M(\un[0])$, which satisfies 
\begin{equation}\label{eq:Avn0}
\dprod{\A[\vn](\vn[n+1]-\vn)}{\vn[n+1]-\vn}\xrightarrow{n\to\infty}0,
\end{equation}
it follows that 
\[
\sup_{\genfrac{}{}{0pt}{}{z\in\X}{\norm{z}=1}}\dprod{\A[\vn]z}{\vn[n+1]-\vn}\to0
\] 
as $n\to\infty$.

\end{enumerate}

\begin{remark}\label{rem:elliptic}
The above conditions (A1)--(A3) are satisfied, in particular, for uniformly bounded and coercive linear operators; i.e., when there exist constants $C_1,C_2>0$ such that
\begin{equation}\label{eq:C1}
\dnorm{\A[w]v}\le C_1\norm{v}\qquad\forall w\in\M(\un[0])\quad\forall v\in\X,
\end{equation}
respectively,
\begin{equation}\label{eq:C2}
\dprod{\A[w]v}{v}\ge C_2\norm{v}^2\qquad\forall w\in\M(\un[0])\quad\forall v\in\X.
\end{equation}
Indeed, in this situation, (A2) is obvious, and (A1) follows from an extension of the Lax--Milgram Theorem, cf.~\cite[Thm.~1]{Hayden1968}. Furthermore, from~\eqref{eq:Avn0} and~\eqref{eq:C2}, we deduce that $\norm{\vn[n+1]-\vn[n]}\to0$ as $n\to\infty$. Then, by virtue of~\eqref{eq:C1}, we immediately obtain that
\[
\sup_{\genfrac{}{}{0pt}{}{z\in\X}{\norm{z}=1}}\dprod{\A[\vn]z}{\vn[n+1]-\vn}
\le C_1\norm{\vn[n+1]-\vn}\xrightarrow{n\to\infty}0,
\]
which is~(A3). Moreover, in that situation, the boundedness conditions~\eqref{eq:Steinhaus} and~\eqref{eq:SteinhausD}, see below, are satisfied as well.
\end{remark}

\begin{proposition}\label{prop:convergence1}
For the sequence generated by the iteration~\eqref{eq:iteration}, we assume that there exists $\mu\in\mathbb{R}$ with $\E(\un)\ge\mu$ for all $n\ge 0$, and that
\begin{equation}\label{eq:Steinhaus}
\sup_{n\ge0}\dnorm{\A[\un]z}<\infty\qquad\forall z\in\X.
\end{equation}
Then, under the conditions {\rm(A1)--(A3)} and {\rm(T)}, the sequence $\{\E(\un)\}_n$ is monotone decreasing (and thus has a limit denoted by $\E^\star$), and $\dnorm{\E'(\un[n])}\to0$ as $n\to\infty$.
\end{proposition}

\begin{proof}
From Proposition~\ref{prop:T} with $z=\un[0]$, we observe that the sequence $\{\E(\un)\}_{n}$ is monotonically decreasing. Hence, since $\E$ is bounded from below, it follows that there exists $\E^\star\ge\mu$ with $\E(\un)\to\E^\star$ as $n\to\infty$, and in particular, $\lim_{n\to\infty}(\E(\un)-\E(\un[n+1]))=0$. Furthermore, owing to~(A2) and~\eqref{eq:Tpropb}, and letting $\delta^n:=\un[n+1]-\un=\T(\un)-\un$, we have that
\begin{align} \label{eq:a3cond}
0
\le
\gamma\dprod{\A[\un]\delta^n}{\delta^n}
\le \E(\un)-\E(\un[n+1]) \xrightarrow{n \to \infty} 0.
\end{align}
Therefore, since $\gamma>0$, cf.~assumption~(T), and by (A2), we observe that the sequence given by
\[
\beta_n:=
\begin{cases}
\sqrt{\dprod{\A[\un]\delta^n}{\delta^n}}&\text{if }\dprod{\A[\un]\delta^n}{\delta^n}>0,\\
\nicefrac{1}{n}&\text{otherwise},
\end{cases}
\qquad n\ge 0,
\]
is positive, and satisfies $\beta_n\to0$ as $n\to\infty$. Now, fix any $z\in\X$ with $\norm{z}=1$. Then, using the positivity from~(A2), we notice that
\[
0\le\beta^{-1}_n\dprod{\A[\un](\delta^n-\beta_nz)}{\delta^n-\beta_nz},
\]
and upon rearranging terms,
\[
\dprod{\A[\un]\delta^n}{z}
\le
\beta_n\left(1+\dprod{\A[\un]z}{z}\right)
-\dprod{\A[\un]z}{\delta^n}.
\]
A similar bound is obtained by replacing $z$ by $-z$, whence we arrive at
\[
\left|\dprod{\A[\un]\delta^n}{z}\right|
\le
\beta_n\left(1+\dprod{\A[\un]z}{z}\right)
+\left|\dprod{\A[\un]z}{\delta^n}\right|.
\]
Furthermore, recalling from~(A2) that $\A[\un]$ is a bounded linear operator for each $n\ge 0$, and applying~\eqref{eq:Steinhaus}, the Banach-Steinhaus theorem yields that 
\[
C:=\sup_{n\ge 0}\dnorm{\A[\un]}<\infty,
\]
and thus,
\[
\dnorm{\A[\un]\delta^n}=
\sup_{\genfrac{}{}{0pt}{}{z\in\X}{\norm{z}=1}}\left|\dprod{\A[\un]\delta^n}{z}\right|
\le
\beta_n\left(1+C\right)
+\sup_{\genfrac{}{}{0pt}{}{z\in\X}{\norm{z}=1}}\left|\dprod{\A[\un]z}{\delta^n}\right|.
\]
Then, letting $n\to\infty$, and invoking~(A3) in the light of~\eqref{eq:a3cond}, we conclude that
\[
\dnorm{\E'(\un[n])}=\dnorm{\A[\un]\delta^n}\xrightarrow{n\to\infty}0,
\]
cf.~\eqref{eq:T}, which completes the argument.
\end{proof}

\subsection{Iterative linearized Galerkin energy reduction and convergence}

On a discrete linear subspace $\X_N\subset\X$, with $\dim(\X_N)<\infty$, for a starting point $\uNn[0]\in\X_N$, we consider the (conforming) discretization of the weak formulation~\eqref{eq:min} given by
\begin{equation}\label{eq:weakW}
u^\star_N\in\M(\uNn[0])\cap\X_N:\qquad \dprod{\E'(u^\star_N)}{v}=0\qquad\forall v\in\X_N.
\end{equation}
In order to be able to solve this problem numerically, we employ the discretized version of the iterative linearization scheme~\eqref{eq:iteration} given by
\begin{equation}\label{eq:iterationW}
\dprod{\A[\uNn](\uNn[n+1]-\uNn[n])}{v}=-\dprod{\E'(\uNn)}{v}\qquad\forall v\in\X_N,\quad n\ge 0.
\end{equation}
Here, for any $w\in\M(\uNn[0])$, we may define the discrete operator 
\begin{equation}\label{eq:AN}
\A_N[w]:\,\X_N\to\X_N',\qquad u\mapsto\A_N[w]u,
\end{equation}
through the weak formulation
\[
\dprod{\A_N[w]u}{v}=\dprod{\A[w]u}{v}\qquad\forall v\in\X_N.
\]
Then, the discrete iteration~\eqref{eq:iterationW} can be written equivalently as
\[
\dprod{\A_N[\uNn](\uNn[n+1]-\uNn[n])}{v}=-\dprod{\E'(\uNn)}{v}\qquad\forall v\in\X_N,\quad n\ge 0.
\]
Similarly, we define a discrete version of the update operator $\T$ from~\eqref{eq:T} by
\begin{equation}\label{eq:TN}
\T_N:\,\M(\uNn[0])\cap\X_N\to\X_N,\qquad
\dprod{\A_N[w](\T_N(w)-w)}{v}=-\dprod{\E'(w)}{v}\quad\forall v\in\X_N.
\end{equation} 
The following lemma shows that our previous hypotheses carry over to the discrete framework.

\begin{lemma}\label{lem:solveN}
The assumptions {\rm(A1)--(A3)} and {\rm (T)} apply to the discrete operators $\A_N$ and $\T_N$ from~\eqref{eq:AN} and~\eqref{eq:TN}, respectively, on any finite-dimensional linear subspace $\X_N\subset\X$; in particular the discrete iteration~\eqref{eq:iterationW} is well-defined.
\end{lemma}

\begin{proof}
For $w\in\M(\uNn[0])$ the assumption (A2) on $\X$ is sufficient for the invertibility of $\A_N[w]$ in the finite dimensional case; i.e. (A1) and (A2) are satisfied on~$\X_N$. Moreover, for $w\in\M(\uNn[0])\cap\X_N$, letting $\delta_w=\T_N(w)-w\in\X_N$, we recall from Remark~\ref{rem:Tequiv} that (T) is equivalent to
\[
\int_0^1\dprod{\E'(s\delta_w+w)-\E'(w)}{\delta_w}\,\d s
\le (1-\gamma)\dprod{\A[w]\delta_w}{\delta_w}
=(1-\gamma)\dprod{\A_N[w]\delta_w}{\delta_w},
\]
which yields the discrete analog to the conditions in~\eqref{eq:Tprop}. Finally, for a sequence $\{\vn\}_n$ in the set $\M(\uNn[0])\cap\X_N$, letting $\delta^n:=\vn[n+1]-\vn\in\X_N$, we suppose that $\lim_{n\to\infty}\dprod{\A_N[\vn]\delta^n}{\delta^n}=0$, cf.~\eqref{eq:Avn0}. Then, in light of~\eqref{eq:AN}, we have $\lim_{n\to\infty}\dprod{\A[\vn]\delta^n}{\delta^n}=0$, and thus from (A3):
\begin{align*}
\sup_{\genfrac{}{}{0pt}{}{z\in\X_N}{\norm{z}=1}}\dprod{\A_N[\vn]z}{\delta^n}
=\sup_{\genfrac{}{}{0pt}{}{z\in\X_N}{\norm{z}=1}}\dprod{\A[\vn]z}{\delta^n}
\le\sup_{\genfrac{}{}{0pt}{}{z\in\X}{\norm{z}=1}}\dprod{\A[\vn]z}{\delta^n}
\xrightarrow{n\to\infty}0.
\end{align*}
This completes the proof.
\end{proof}

In order to approximate a solution of the original equation~\eqref{eq:min}, we consider a hierarchical sequence of finite-dimensional Galerkin spaces $\{\X_N\}_{N} \subset \X$; i.e, we have that 
\[
\X_0 \subset \X_1 \subset \X_2\subset\dotsc \subset \X_N\subset\dotsc \subset \X.
\]
We will suppose that the sequence of subspaces is rich enough in the following sense:

\begin{enumerate}[(A)]
\item[(D)] \emph{Approximability in Galerkin subspaces:} 
For any $w \in \X$, there exists a sequence $w_N \in \X_N$, $N \in \mathbb{N}$, which converges strongly to $w$; i.e., we have that $
\norm{w-w_N} \to 0$ as $N \to \infty$.
\end{enumerate}

Then, we will perform sufficiently many discrete iteration steps~\eqref{eq:iterationW} on each of the subspaces $\X_N$ to generate a sequence of approximations that potentially converges to a solution of~\eqref{eq:min}; this which will be made precise in Theorem~\ref{thm:convergence}. For the proof of that theorem, we assume further that the following conditions are satisfied:
\begin{enumerate}[(A)]
\item[(E)] \emph{Residual convergence:} Whenever a weakly converging sequence $\{v^n\}_n \subset \X$, with a weak limit point $v^\star \in \X$, i.e. $v^n\rightharpoonup v^\star$ as $n\to\infty$, satisfies the two residual-related conditions
\begin{equation}\label{eq:Eprop}
\lim_{n\to\infty}\dprod{\E'(v^n)}{w}=0 \quad\forall w\in\X\qquad\text{and}\qquad
\lim_{n\to\infty}\dprod{\E'(v^n)}{v^n}=0,
\end{equation}
then it follows that $\E'(v^\star)=0$ in $\X'$.\medskip

\item[(B)] \emph{Local boundedness:} For any $w \in \X$ and any $\varrho>0$ there exists $R \geq 0$ (depending on $w$ and $\varrho$) such that $\dnorm{\E'(v)} \leq R $ for any $v \in \X$ with $\norm{w-v} \leq \varrho$.\medskip
\end{enumerate}

The ensuing theorem states that our energy-based approach generates a computable sequence that converges to a solution of the original problem~\eqref{eq:min}. 

\begin{theorem}[Convergence of the iterative linearized Galerkin energy reduction scheme]
\label{thm:convergence}
Suppose that the energy functional $\E$ satisfies the properties {\rm (A1)--(A3)} and {\rm (T)} on the (continuous) space $\X$. Moreover, for any (finite) $N \geq 0$, assume that the sequence $\{\uNn\}_{n}$ (or at least a subsequence thereof) generated by the iteration~\eqref{eq:iterationW} as well as the associated sequence of energies $\{\E(\uNn)\}_n$ both stay bounded, and that
\begin{align} \label{eq:SteinhausD}
\sup_{n\ge0}\|\A[\uNn]z\|_{\X_N'}<\infty\qquad\forall z\in\X_N.
\end{align}
Moreover, consider a positive sequence $\{\epsilon_N\}_{N\ge 0}$ with $\lim_{N\to\infty}\epsilon_N= 0$. Then, for each $N\ge 0$, there exists an index $n^\star=n^\star(N)$ such that $\uNn[n^\star]\in\X_N$ generated after $n^\star$ steps of the discrete iteration~\eqref{eq:iterationW} satisfies
\begin{equation}\label{eq:E'small}
\big\|\E'(\uNn[n^\star])\big\|_{\X_N'} \leq \epsilon_N.
\end{equation}
Furthermore, if the conditions {\rm (E), (B)} and {\rm(D)} are satisfied, and provided that the sequence of all final iterates $\{\uNn[n^\star]\}_{N\ge 0}$ is bounded, then there exists a weakly converging subsequence $\{u^{n^\star}_{N'}\}_{N'}\subset\X$, with a weak limit $u^\star\in\X$ that is a solution of the weak formulation~\eqref{eq:min}. 
\end{theorem}

\begin{proof}
By Lemma~\ref{lem:solveN}, the properties {\rm (A1)--(A3)} and {\rm (T)} transfer to any linear subspace $\X_N$ of $\X$. Hence, for each $N\ge0$, we infer from Proposition~\ref{prop:convergence1} that 
\[
\Norm{\E'(u_N^n)}_{\X_N'} \to 0 \quad \text{as} \ n \to \infty,
\]
and, in turn,~\eqref{eq:E'small} is satisfied for $n^\star$ large enough.
We continue by verifying that the conditions from property (E) are satisfied for a subsequence of $\{\uNn[n^\star]\}_{N}$. Indeed, due to the boundedness of the sequence $\{\uNn[n^\star]\}_{N}$, and by virtue of the reflexivity of the Banach space~$\X$, we can extract a weakly converging subsequence $\{u^{n^\star}_{N'}\}_{N'}$ in $\X$ with a weak limit $u^{n^\star}_{N'}\rightharpoonup u^\star\in\X$. Then, letting
\[
\rho:=\sup_{N'}\norm{u^\star-u^{n^\star}_{N'}}<\infty,
\]
and applying property (B), there exists $R>0$ (depending on $u^\star$ and on $\rho$) such that 
\[
\big\|\E'(u^{n^\star}_{N'})\big\|_{\X'}\le R.
\]
Now fix an arbitrary element $w\in\X$ with $\norm{w}\le 1$, and consider a sequence $\{w_N\}_{N}$ with $w_N \in \X_N$ for all $N \geq 0$ that converges strongly to $w$, cf.~property (D). Then, there exists $K \geq 0$ such that $\norm{w_N}\le K$ for all $N \geq 0$. In turn, we have that
\begin{align*}
\left|\dprod{\E'(u_{N'}^{n^\star})}{w}\right|
&\le\left|\dprod{\E'(u_{N'}^{n^\star})}{w_{N'}}\right|
+\left|\dprod{\E'(u_{N'}^{n^\star})}{w-w_{N'}}\right|\\
&\le\big\|\E'(u_{N'}^{n^\star})\big\|_{\X_N'}\norm{w_{N'}}
+\big\|\E'(u_{N'}^{n^\star})\big\|_{\X'}\norm{w-w_{N'}}\\
&\le K\epsilon_{N'}
+R \norm{w-w_{N'}},
\end{align*}
and thus 
$
\lim_{N' \to \infty} \dprod{\E'(u_{N'}^{n^\star})}{w}=0$; i.e., the subsequence $\{u_{N'}^{n^\star}\}$ fulfills the first property in~\eqref{eq:Eprop}. It remains to verify the second limit in~\eqref{eq:Eprop}, for which we recall that $\{u_{N'}^{n^\star}\}_{N'}$ is bounded (due to the weak convergence); i.e., there exists $M>0$ such that $\norm{u_{N'}^{n^\star}} \leq M$ for all $N'$. Together with~\eqref{eq:E'small} this leads to
\[
0 \leq \left|\dprod{\E'(u_{N'}^{n^\star})}{u_{N'}^{n^\star}} \right| \leq \Norm{\E'(u_{N'}^{n^\star})}_{\X_{N'}'} \norm{u_{N'}^{n^\star}} \leq \epsilon_{N'} M \xrightarrow{N' \to \infty} 0.
\]
Finally, we are able to employ property (E), which yields $\E'(u^\star)=0$ in $\X'$.
\end{proof}

We will provide two remarks on the boundedness assumptions in the previous theorem. In practice, we will use the final iteration in $\X_N$ as our initial guess in $\X_{N+1}$, i.e., we set $u_{N+1}^0=u_{N}^{n^\star}$. Consequently, the sequence of energies $\{\E(u_N^{n^\star})\}_N$ is decreasing.

\begin{remark} \label{rem:boundedseq}
If $\E$ is weakly coercive on $\X$, i.e., $\E(v)\to+\infty$ whenever $\norm{v}\to\infty$,  then the set $\M(v^0)$ from~\eqref{eq:M} is bounded. Indeed, if this were not true, then we could extract an unbounded sequence $\{v^n\}_n$ in $\M(v^0)$, which, by the weak coercivity of $\E$, would satisfy $\E(v^n) > \E(v^0)$ for $n$ large enough; this, however, contradicts the definition of the set~$\M(v^0)$. In particular, it follows that any sequence $\{v^n\}_n \subset \X$ with a decreasing sequence of associated energies $\{\E(v^n)\}_n$ is bounded.
\end{remark}

\begin{remark} \label{rem:boundedenergie}
Provided that $\E$ is weakly sequentially lower semicontinuous and $\{v^n\}_n$ is any bounded sequence in $\X$, then $\{\E(v^n)\}_n$ is bounded from below. Otherwise, we could extract a subsequence $\{v^{n'}\}_{n'} \subset \X$ with 
\[
\E(v^{n'}) \to -\infty \quad \text{for} \ n' \to \infty.
\]
Since $\{v^{n'}\}_{n'}$ is a bounded sequence in the reflexive space $\X$, we can further extract (a not relabelled) subsequence $\{v^{n'}\}_{n'}$ with a weak limit point $v^\star$. Then, by the weak lower semicontinuity we find that
\[
- \infty < \E(v^\star) \leq \lim_{n' \to \infty} \E(v^{n'}) \xrightarrow{n' \to \infty} -\infty,
\]
a contradiction.
\end{remark}

\subsection{Abstract examples}

In this section, we verify property (E) in the context of two abstract problem settings that are relevant in many practical applications; some specific examples of semilinear and quasilinear diffusion-reaction models will be discussed later on in \S\ref{sec:drmodels}.

\subsubsection{Compactly perturbed semilinear problems}

Consider a Banach space $\Y$ with compact embedding $\iota:\,\X\comp\Y$. Suppose that $\E$ has a derivative of the form
\begin{equation}\label{eq:Esplit}
\E'(v)=\operator{L}(v)+g(v),\qquad v\in\X,
\end{equation}
where $\operator{L}:\,\X\to\X'$ is a bounded linear mapping, and $g:\,\X\to\X'$ is weakly continuous in the sense that $g(v^n)\wstar g(v^\star)$ for any sequence $\{v^n\}_n\subset\X$ that converges to a limit $v^\star\in\X$ strongly in~$\Y$, i.e.
\begin{equation}\label{eq:G}
\lim_{n\to\infty}\dprod{g(v^n)}{w}=\dprod{g(v^\star)}{w}\qquad\forall w\in\X,
\end{equation}
whenever $\lim_{n\to\infty}\|v^n-v^\star\|_{\spc{Y}}=0$.

\begin{lemma}[Property (E)]\label{lem:compE}
If the energy functional $\E$ exhibits the structure~\eqref{eq:Esplit} and condition~\eqref{eq:G} holds, then property {\rm (E)} is fulfilled.
\end{lemma}

\begin{proof}
For any given $w\in\X$, define the bounded linear functional $\lambda_w\in\X'$ by
\[
\lambda_w(v):=\dprod{\operator{L}(v)}{w},\qquad v\in\X.
\]
Then, for any weakly converging sequence $v^n\rightharpoonup v^\star$ as $n\to\infty$, we have
 \begin{equation}\label{eq:help1perp}
 \lim_{n\to\infty}\dprod{\operator{L}(v^{n})-\operator{L}(v^\star)}{w}
 =\lim_{n\to\infty}\lambda_w(v^n-v^\star)
 =0.
 \end{equation}
 Moreover, by the compact embedding $\iota:\,\X\comp\spc{Y}$, there exists a subsequence $\{\iota(v^{n'})\}_{n'}$ that converges strongly in $\spc{Y}$ to a limit $w^\star\in\spc{Y}$. Then, for any $\ell\in\spc{Y}'$, it holds that $\ell(\iota(v^{n'}))\to\ell(w^\star)$ for $n'\to\infty$. Furthermore, observing that $\ell\circ\iota\in\X'$, and due to the weak convergence of $\{v^{n'}\}_{n'}$, we have
 \begin{align*}  \ell\circ\iota(v^{n'})\xrightarrow{n'\to\infty}\ell\circ\iota(v^\star).
 \end{align*}
We conclude that $\ell(w^\star)=\ell\circ\iota(v^\star)$ for any $\ell\in\Y'$, and thus, owing to (a consequence of) the Hahn-Banach theorem and by canonical identification, we deduce $w^\star=v^\star$.
Hence, invoking~\eqref{eq:G}, for any $w\in\X$, we have that
 \begin{align} \label{eq:help2perp}
 \lim_{n'\to\infty}\dprod{g(v^{n'})-g(v^\star)}{w}\to 0.
 \end{align}
Moreover, using the representation~\eqref{eq:Esplit}, for any $w\in\X$, it follows that
\begin{align*}
0 \leq \left|\dprod{\E'(v^\star)}{w}\right|
&\le \left|\dprod{\E'(v^{n'})}{w}\right|+\left|\dprod{\E'(v^{n'})-\E'(v^\star)}{w}\right|\\
&\le\left|\dprod{\E'(v^{n'})}{w}\right|+\left|\dprod{\operator{L}(v^{n'})-\operator{L}(v^\star)}{w}\right|+\left|\dprod{g(v^{n'})-g(v^\star)}{w}\right|.
\end{align*}
If $\E'(v^{n'})\wstar 0$ for $n'\to\infty$, then applying~\eqref{eq:help1perp} and \eqref{eq:help2perp}, we conclude that $\dprod{\E'(v^\star)}{w}=0$ for all $w\in\X$.
\end{proof}

\subsubsection{Monotone problems}

Next, we consider the case where the G\^{a}teaux derivative of the energy functional $\E: \X \to \mathbb{R}$ involves some monotonicity property (on the full space $\X$). Specifically, we assume that $\E'$ is \textit{monotone} in the sense that
\begin{equation}\label{eq:monotone}
\dprod{\E'(u)-\E'(v)}{u-v} \geq 0 \qquad \forall u,v \in \X;
\end{equation}
we remark that $\E'$ is called strictly monotone if the inequality in~\eqref{eq:monotone} is strict for $u \neq v$.

\begin{lemma} \label{lem:monotone}
If $\E: \X \to \mathbb{R}$ is convex (i.e. $\E':\,\X\to\X'$ is a monotone operator), then $\E$ is weakly sequentially lower semicontinuous and satisfies property {\rm (E)}. Furthermore, if $\E'$ is strongly monotone, meaning that there is a constant $\eta>0$ such that
\begin{align} \label{eq:strongmon}
\dprod{\E'(u)-\E'(v)}{u-v} \geq \eta \norm{u-v}^2 \qquad \forall u,v \in \X,
\end{align}
then $\E$ is also weakly coercive.
\end{lemma}

\begin{proof}
For any $v,w\in\X$, employing the fundamental theorem of calculus together with the monotonicity~\eqref{eq:monotone} results in
\begin{align}\label{eq:Efund}
    \E(v)-\E(w)
    &=\dprod{\E'(w)}{v-w}+\int_0^1\dprod{\E'(sv+(1-s)w)-\E'(w)}{v-w}\ds
    \ge\dprod{\E'(w)}{v-w}.
\end{align}
Hence, for a weakly converging sequence $v^n\xrightharpoonup{n\to\infty}v^\star$ in $\X$, we have that
\[
\E(v^\star)+\dprod{\E'(v^\star)}{v^n-v^\star}\le\E(v^n),
\]
which, for $n\to\infty$, leads to $\E(v^\star)\le\liminf_{n\to\infty}\E(v^n)$; i.e., $\E$ is weakly sequentially lower semicontinuous. Furthermore, for any $w \in \X$, we find by~\eqref{eq:Efund} that
\[
\E(w) \geq \E(v^n) + \dprod{\E'(v^n)}{w-v^n} \geq \E(v^n)+\dprod{\E'(v^n)}{w} - \dprod{\E'(v^n)}{v^n}.
\]
Given the assumptions of property (E), both dual products on the righ-hand side vanish in the limit $n \to \infty$. Consequently, recalling the weak lower semicontinuity, we find that
\[
\E(w) 
\geq \limsup_{n \to \infty} \E(v^n)
\geq \liminf_{n \to \infty} \E(v^n)
\geq \E(v^\star)\qquad \forall w\in\X.
\]
In particular, $v^\star$ is a global minimiser of $\E$, and in turn $\E'(v^\star)=0$ in $\X'$; cf.~\cite[Prop. 25.11]{Zeidler:90}. Finally, we will show that strong monotonicity implies weak coercivity. To this end, for any $v,w \in \X$, recalling~\eqref{eq:Efund}, we find that
\begin{align}
\E(v) & \geq \E(w) + \dprod{\E'(w)}{v-w} + \eta \int_0^1 s \norm{v-w}^2 \ds \nonumber\\
& \geq \E(w)- \dnorm{\E'(w)}\norm{v-w}+\frac{\eta}{2} \norm{v-w}^2,\label{eq:Evbd}
\end{align}
and therefore $\E(v) \to \infty$ for $\norm{v} \to \infty$ (and fixed $w\in\X$).
\end{proof}

\begin{remark}\label{rem:Mbd}
Under the strong monotonicity condition~\eqref{eq:strongmon}, for any $\un[0]\in\X$, we note that the set $\M(\un[0])$ from~\eqref{eq:M} is bounded. Indeed, for any $w\in\M(\un[0])$, proceeding as in~\eqref{eq:Evbd}, we have that
\begin{align*}
0\ge\E(w)-\E(\un[0])
\ge-\dnorm{\E'(\un[0])}\norm{w-\un[0]}+\frac{\eta}{2}\norm{w-\un[0]}^2,
\end{align*}
which leads to
\[
\norm{w}
\le \norm{\un[0]}+\norm{w-\un[0]}\le\norm{\un[0]}+\frac{2}{\eta}\dnorm{\E'(\un[0])},
\]
and hence, the set $\M(\un[0])$ is bounded. In particular, the boundedness requirement on the sequence $\{\uNn\}_n$ from Theorem~\ref{thm:convergence} is satisfied for strongly monotone problems. In addition, the energy sequence $\{\E(\uNn)\}_n$ is bounded from below by the weak sequential lower semicontinuity of $\E$, cf.~Remark~\ref{rem:boundedenergie}, and bounded from above by the definition of the set $\M(\un[0])$ provided that property~(T) holds.
\end{remark}

Finally, within the above setting, we observe that the full sequence in Theorem~\ref{thm:convergence} converges (strongly) to a solution of the weak formulation~\eqref{eq:min}.

\begin{proposition} \label{prop:strongconvergence}
If $\E$ in Theorem~\ref{thm:convergence} is strictly convex (i.e., $\E'$ is strictly monotone), then the full sequence $\{u_{N}^{n^\star}\}_N$ converges to a weak limit $u^\star$. Furthermore, if $\E$ is even strongly convex, cf.~\eqref{eq:strongmon}, then the sequence converges strongly.
\end{proposition}

\begin{proof}
By the very same argument as in the proof of Lemma~\ref{lem:monotone}, we have that the weak limit $u^\star$ is a global minimiser of $\E$. Next, we will show that the minimiser is unique. Suppose by contradiction that there were another minimiser $\widetilde u^\star \neq u^\star$ of $\E$. Then, for $w:=\nicefrac12(u^\star+\widetilde u^\star)$, exploiting strict convexity of $\E$ results in
\[
\E(w)<\frac12\E(u^\star)+\frac12\E(\widetilde u^\star)=\min_{v\in\X}\E(v),
\]
a contradiction. In addition, from strict monotonicity, we observe that $u^\star$ is the only critical point of $\E$. Hence, any weakly converging subsequence $\{u_{N_k}^{n^\star}\}_{k}$ has the same limit, whence the full sequence converges weakly to $u^\star$. Finally, if we presume strong convexity of $\E$, i.e., if \eqref{eq:strongmon} is satisfied, then employing~\eqref{eq:Efund} and \eqref{eq:strongmon}, we find that
\[
0 \leq \norm{u^\star-u_N^{n^\star}}^2 \leq \frac{2}{\eta}\left(\E(u_N^{n^\star})-\E(u^\star)\right) \xrightarrow{N \to \infty} 0, 
\]
where we have used that $\E'(u^\star)=0$ in $\X'$.
\end{proof}

\section{Variational adaptivity} \label{sec:VA}
Our notion of \emph{variational adaptivity} refers to an iterative procedure that drives the numerical solution process for the continuous problem~\eqref{eq:min} by exploiting its underlying energy (i.e.~variational) structure. This will be the focal point of the present section, where we will present and analyse an \emph{energy-based} adaptive algorithm that employs an interplay of the iterative linearization procedure~\eqref{eq:iteration} and abstract adaptive Galerkin discretizations thereof. Specifically, we will provide a computational scheme that is able to realize the sequence $\{u_N^{n^\star}\}_N$ from Theorem~\ref{thm:convergence}. To this end, we will first give a general outline of the approach in the following \S\ref{sec:components}, before turning to a more specific framework in the context of finite element discretizations in \S\ref{sc:VAFEM}; for the presentation of the latter, we will limit ourselves to a brief exposition by following closely along the lines of~\cite[\S4]{AHW:23}, with some modifications to be applied.

\subsection{Components of the energy-based adaptive procedure} \label{sec:components}

The proposed adaptive strategy, see Algorithm~\ref{alg:AILFEM}, is based on three components, which will be outlined briefly in the sequel.

\subsubsection*{{\rm(i)} Iterative linearization.} On a given discrete subspace $\X_N$ we perform the iterative linearization scheme~\eqref{eq:iteration} until the discrete approximation $u_N^n \in \X_N$ is close enough to a solution
of the corresponding discrete weak problem~\eqref{eq:weakW}. In order to monitor the approximation quality, we will consider the (computable) discrete residual $\RN(u_N^n)$, where we let
\begin{align} \label{eq:discreteRed}
\RN(v):=\|\E'(v)\|_{\X_N'}, \qquad v \in \X_N.
\end{align}
In particular, we stop the iteration~\eqref{eq:iterationW} on the current Galerkin space $\X_N$ as soon as $\RN(u_N^n)$ is deemed
small enough; the final iteration number on the space $\X_N$ is denoted, as before, by $n^\star=n^\star(N)$.

\subsubsection*{{\rm(ii)} Adaptive Galerkin space enrichments.} The construction of the discrete Galerkin spaces $\{\X_N\}_N$ is
based on the assumption that we have at our disposal an adaptive strategy that is able to identify some local information on the error, and thereby to appropriately enrich the current Galerkin
space $\X_N$ once the discrete residual is sufficiently small, cf.~(i). Subsequently, we set $u_{N+1}^0:=u_{N}^{n^\star}$. This component of the algorithm will be discussed more specifically for finite element discretizations in \S\ref{sc:VAFEM} below.

\subsubsection*{{\rm(iii)} Decision between linearization and discretization.} 
On a given discrete subspace $\X_N$, we introduce the total energy reduction after $n$ iteration steps of the energy reduction procedure~\eqref{eq:iterationW} by
\begin{align} \label{eq:GFIincrement}
	\Delta\E^n_N:= \E(u_N^{0}) -  \E(u_N^{n}).
\end{align}
Since the energy decays along the generated sequence, cf.~Proposition~\ref{prop:convergence1}, the energy reduction $\Delta\E_N^n$ is non-negative and increasing in $n$. On the other hand, the same proposition yields that the discrete residual $\RN(u_N^n)$ vanishes as $n \to \infty$. Consequently, for any bounded sequence of positive numbers $\{\gamma_N\}_N$ associated to the sequence of discrete subspaces, we have that
\begin{align} \label{eq:stoppingcriterion}
\RN(u_N^n) \leq \gamma_N \Delta\E^n_N,
\end{align}
if $n$ (depending on $N$) becomes large enough; this bound will be employed as our stopping criterion. Assuming, as in Theorem~\ref{thm:convergence}, that $\{\E(u_N^{n^\star})\}_N$ is bounded from below, we notice that the energy sequence converges, and thus,
\[
0 \leq \Delta \E_N^{n^\star(N)} =\E(u_{N}^0)-\E(u_N^{n^\star(N)})= \E(u_{N-1}^{n^\star(N-1)})-\E(u_N^{n^\star(N)}) \to 0, \qquad N \to \infty.
\]
Therefore, for the right-hand side in~\eqref{eq:stoppingcriterion}, we can set $\epsilon_N:=\gamma_N \Delta \E_N^{n^\star}$ in~\eqref{eq:E'small}.
\medskip

The above components (i)--(iii) give rise to an energy-based adaptive iterative linearized Galerkin scheme; see Algorithm~\ref{alg:AILFEM}.

\begin{algorithm}
\begin{algorithmic}[1]
\State Start with an initial Galerkin space $\X_0$ and an initial guess $u^{0}_0 \in \X_0$.
\State Set $N=n=0$.
\For{$N=0,1,2,\ldots$ and some bounded values $\gamma_N>0$}
\Repeat 
\State Perform one iterative step~\eqref{eq:iterationW} in $\X_N$ to obtain $u_N^{n+1}$ from $u_N^n$.
\State Update $n \leftarrow n+1$.
\Until{$\RN(u_N^n) \leq \gamma_N \Delta\E_N^n$}
\State Set $u_N^{n^\star}:=u_N^n$.
\State Enrich the Galerkin space $\X_N$ appropriately, cf.~(ii) above, in order to obtain $\X_{N+1}$.
\State Define $u_{N+1}^0:=u_N^{n^\star}$ by canonical embedding $\X_N \hookrightarrow \X_{N+1}$. 
\State Update $N \leftarrow N+1$ and $n \leftarrow 0$.
\EndFor
\State \textsc{Return} the sequence $\{u_N^{n^\star}\}_{N}$.
\end{algorithmic}
\caption{Energy-based adaptive iterative linearized Galerkin procedure}
\label{alg:AILFEM}
\end{algorithm}

In practice, a stopping criterion in line~3 of Algorithm~\ref{alg:AILFEM} is required; this may be expressed, for instance, in terms of a prescribed maximum number of space refinements, or by monitoring the leveling off of the (monotone decreasing) energy values over the sequence of adaptively generated Galerkin spaces; see~\cite{LeuWihler:25} for a more detailed study. The following result, which follows immediately from Theorem~\ref{thm:convergence} and the previous discussion in~(iii), yields convergence in the asymptotic regime $N\to\infty$.

\begin{corollary}
Suppose that the properties {\rm (A1)--(A3), (B), (D), (T)} and {\rm (E)} are satisfied. Moreover, assume that the sequence $\{\uNn\}_{n}$ (or at least a subsequence thereof) generated by the iteration~\eqref{eq:iterationW} as well as the associated sequence of energies $\{\E(\uNn)\}_n$ both stay bounded for any $N \geq 0$, and that
\[
\sup_{n\ge0}\|\A[\uNn]z\|_{\X_N'}<\infty\qquad\forall z\in\X_N.
\]
Then there exists a weakly converging subsequence $\{u^{n^\star}_{N'}\}_{N'}$ of the sequence generated by Algorithm~\ref{alg:AILFEM} with a weak limit $u^\star\in\X$ that is a solution of the weak formulation~\eqref{eq:min}.
\end{corollary}

\subsection{Variational adaptivity for $\mathbb{P}_{p}$ finite element discretizations} \label{sc:VAFEM}

In the context of the finite element method, for adaptive Galerkin space enrichments, cf.~(ii) in \S\ref{sec:components}, standard a posteriori error estimators could be employed. Alternatively, for the purpose of the present work, we will exploit the underlying energy structure induced by the variational framework; specifically, we will leverage the novel energy-based mesh refinement strategy developed in the related works~\cite{HeidStammWihler:19, HeidJCP, AHW:23, HHSW:25} in our experiments in \S\ref{sec:NE} below.

For $N\ge 0$, let $\mathcal{T}_N$ be a conforming and shape-regular partition of the domain $\Omega$ into simplicial elements $\{\kappa\}_{\kappa \in \mathcal{T}_N}$ (i.e. triangles for $d=2$ and tetrahedra for $d=3$). For a fixed polynomial degree $p \geq 1$, we consider the corresponding finite element space 
\[
\X_N=\{v \in \X: \ v|_{\kappa} \in \mathbb{P}_p(\kappa), \ \kappa \in \mathcal{T}_N\},
\]
where $\mathbb{P}_p$ denotes the space of polynomials of (total) degree at most $p$. Moreover, for any element $\kappa \in \mathcal{T}_N$, we denote by $\omega_{\kappa}$ the element patch consisting of $\kappa$ and its facewise neighbors. We further signify by~$\widetilde{\omega}_\kappa$ the locally refined patch obtained by a red refinement of $\kappa$ and by removing any introduced hanging nodes in $\omega_\kappa$ by subsequent green refinements; see Figure~\ref{fig:redgreen} for the two-dimensional situation.
\begin{figure}
\begin{center}
\begin{tabular}{cc}
\includegraphics[scale=0.2]{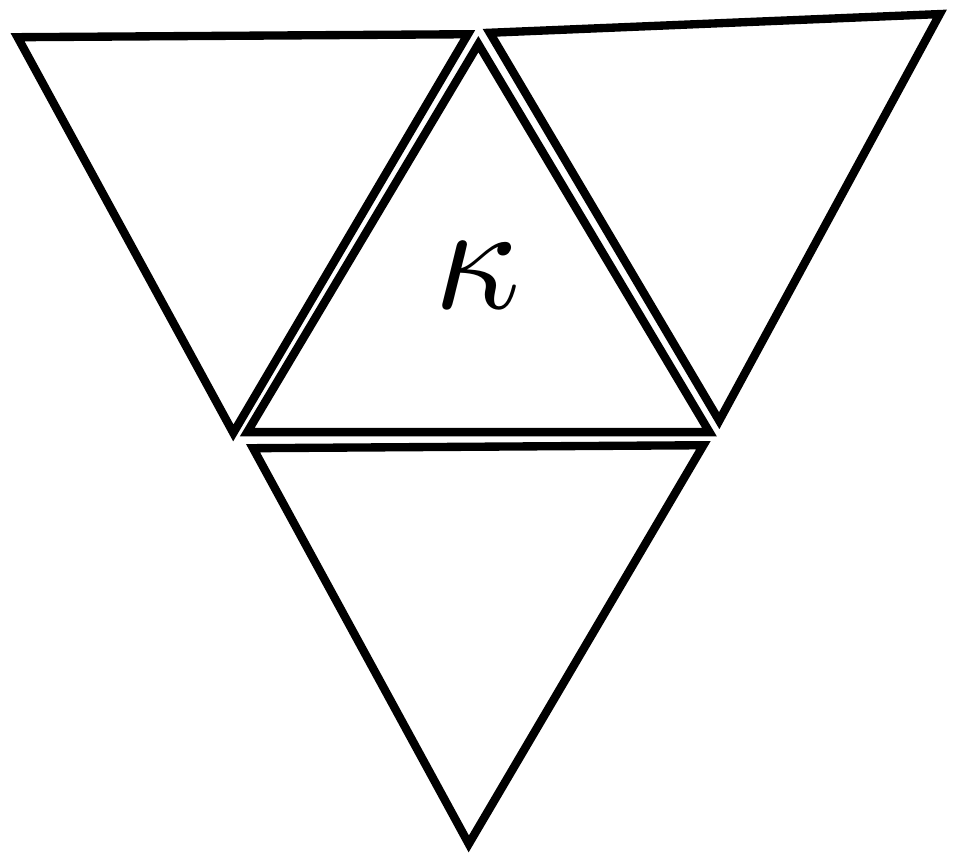} &
\includegraphics[scale=0.2]{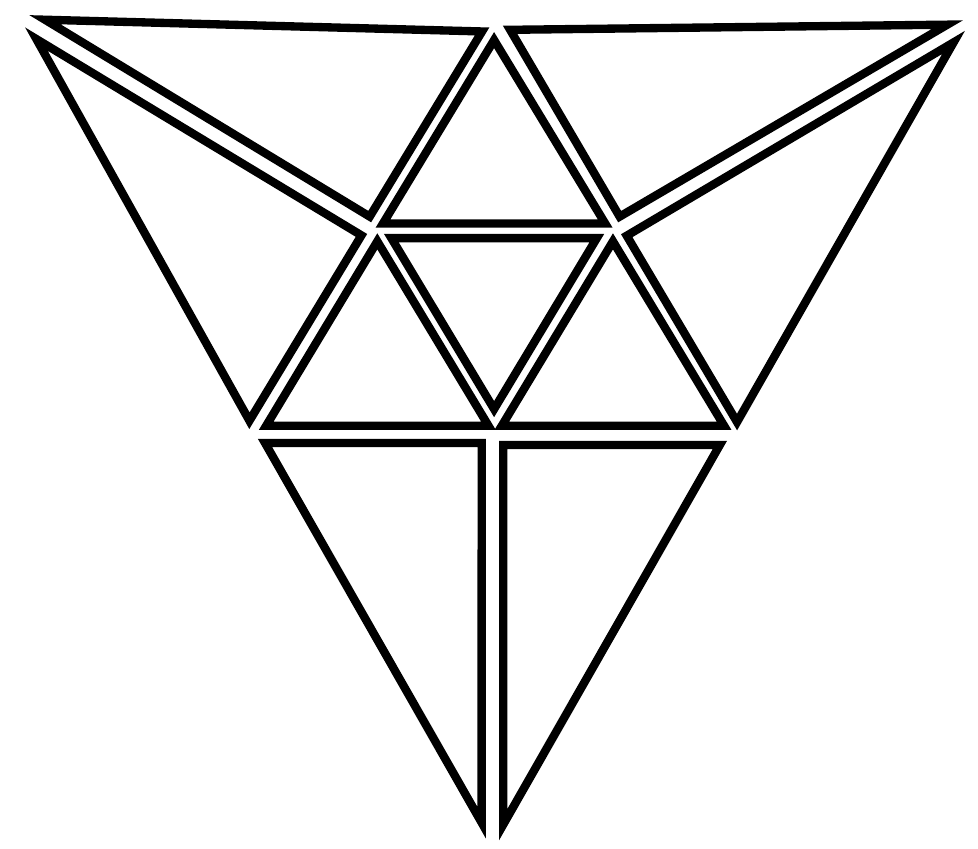}
\end{tabular}
\end{center}
\caption{Local element patches associated to a triangular element~$\kappa$. Left: Mesh patch~$\omega_\kappa$ consisting of the element $\kappa$ and its facewise neighbours. Right: Locally refined patch $\widetilde{\omega}_\kappa$ constructed based on red refining $\kappa$ and on green refining its facewise neighbours.}
\label{fig:redgreen}
\end{figure}
Moreover, let $\{\xi_\kappa^1, \dotsc, \xi_\kappa^{m_\kappa}\}$ be a basis of the local finite element space
\[
\mathbb{V}(\widetilde{\omega}_\kappa):=\{v \in \X: \ v|_\zeta \in \mathbb{P}_p(\zeta) \ \forall \zeta \in \widetilde{\omega}_\kappa \ \text{and} \ v|_{\Omega \setminus \widetilde{\omega}_\kappa} \equiv 0\}.
\]
For any finite element approximation $u_N^n \in \X_N$ and $\kappa \in \mathcal{T}_N$, we further consider the extended space
\begin{equation}\label{eq:Vext}
\mathbb{V}(\widetilde{\omega}_\kappa; u_N^n)=\Span\{\xi^1_\kappa,\ldots,\xi^{m_{\kappa}}_\kappa,u_N^n\} \subset \X,
\end{equation}
which contains one additional \emph{global} degree of freedom. Performing a discrete iteration step~\eqref{eq:iterationW} in the space $\mathbb{V}(\widetilde{\omega}_\kappa; u_N^n)$, we obtain a potentially improved approximation $\widetilde{u}_{N,\kappa}^n$. In view of Lemma~\ref{lem:solveN} and Proposition~\ref{prop:T} (with the initial guess chosen to be $u_N^n$ on the discrete space $\mathbb{V}(\widetilde{\omega}_\kappa; u_N^n)$), we have that
\[\E(\widetilde{u}_{N,\kappa}^n) \leq \E(u_N^n),\]
and, in turn,
\begin{align} \label{eq:deltaE}
\Delta \E_N^n(\kappa):=\E(u_N^n)-\E(\widetilde{u}_{N,\kappa}^n) \geq 0.
\end{align}
This value indicates a potential energy reduction through the refinement of the (single) element $\kappa \in \mathcal{T}_N$. Based on this observation, we employ a D\"orfler type marking strategy to identify and refine a suitable set of elements with the highest local energy reduction indicators; we outline this adaptive strategy in Algorithm~\ref{alg:refen}, whereby we note that additional elements may be refined to preserve the shape-regularity of the adaptively generated mesh sequence $\{\mathcal{T}_N\}_N$.

\begin{algorithm}
\caption{Energy-based adaptive finite element mesh refinement}
\label{alg:refen}
\begin{algorithmic}[1]
\State Prescribe a mesh refinement parameter~$\theta\in(0,1)$. 
\State Input a finite element mesh~$\mathcal{T}_N$ and a finite element function $u^{n}_N \in \X_N$.
\For {all elements $\kappa\in\mathcal{T}_N$}
	\State\multiline{\textsc{Solve} one local discrete iteration step~\eqref{eq:iterationW} in the low-dimensional space $\mathbb{V}(\Pref_\kappa;u_N^n)$, cf.~\eqref{eq:Vext}, to obtain a potentially improved local approximation~$\widetilde u_{N,\kappa}^n$.}  
	\State \textsc{Compute} the local energy reduction~$\Delta \E_N^n(\kappa)$ from~\eqref{eq:deltaE}.    
\EndFor
\State \textsc{Mark} a subset ~$\mathcal{K} \subset \mathcal{T}_N$ of minimal cardinality according to the D\"orfler marking criterion:
\[
\sum_{\kappa \in \mathcal{K}} \Delta \E_N^n(\kappa) \geq \theta \sum_{\kappa \in \mathcal{T}_N} \Delta \E_N^n(\kappa).
\]
\State \textsc{Refine} all elements in~$\mathcal{K}$, and generate a new mesh~$\mathcal{T}_{N+1}$.
\end{algorithmic}
\end{algorithm}

\begin{remark}
We note that there is no guarantee that the sequence of finite element spaces $\{\X_N\}_N$ corresponding to the sequence of meshes $\{\mathcal{T}_N\}_N$ generated by Algorithm~\ref{alg:refen} satisfies property (D). However, if we employ a uniform mesh refinement after a fixed number $K\ge 1$ of local mesh refinements, then property (D) is indeed satisfied. However, for practical purposes, we will not employ this strategy in our numerical experiments below.
\end{remark}

\section{Application to nonlinear diffusion-reaction models} \label{sec:drmodels}

We illustrate the abstract theory in \S\ref{sec:IER} by verifying the properties (A1)--(A3), (T), (B) and (E), for different classes of nonlinear elliptic boundary value problems. Specifically, on a bounded domain $\Omega\subset\R^d$, $1\le d\le 3$, with a Lipschitz boundary~$\partial\Omega$, for some given coefficient $\alpha>0$, we consider the energy functional
\begin{equation}\label{eq:diffre}
\E(v):=\frac{\alpha}{2}\int_\Omega \Psi \left(|\nabla v(\x)|^2\right)\dx-\alpha\int_\Omega\F(\x,v(\x))\dx,
\end{equation}
on the Sobolev space $\X:=\Hone$ of all weakly differentiable functions with first-order partial derivatives in $\Ltwo$, and vanishing Dirichlet trace along the boundary~$\partial\Omega$, equipped with the standard norm $\|\blank\|_{\Hone}:=\|\nabla(\blank)\|_{\Ltwo}$. Here $\Psi:\,[0,\infty)\to\R$ is a continuously differentiable function whose derivative $\psi:=\Psi'$ satisfies the monotonicity bounds
\begin{align} \label{eq:mu}
m_\psi(t-s) \leq \psi(t^2)t-\psi(s^2)s
 \leq M_\psi (t-s) \qquad \forall t\geq s \geq 0,
\end{align}
for two constants $M_\psi\ge m_\psi>0$; in particular, choosing $s=0$ and $t>0$, this condition implies that 
\begin{equation}\label{eq:psi'}
m_\psi\le\psi(t^2)\le M_\psi\qquad \forall t\in\R.
\end{equation} 
Moreover, $\F$ is a potential that satisfies
\begin{equation}\label{eq:f}
\partial_z\F(\cdot,z)=f(\cdot,z)\qquad \text{a.e. in }\Omega,
\end{equation}
for a function $f:\,\Omega\times\R\to\R$. The weak formulation of the associated Euler--Lagrange equation reads:
\begin{equation}\label{eq:E'}
u\in\Hone:\qquad \dprod{\E'(u)}{v}=\alpha\int_\Omega \psi(|\nabla u|^2)\nabla u\cdot\nabla v\dx
-\alpha\int_\Omega f(\cdot,u)v\dx=0,
\end{equation}
for all $v\in\Hone$.\medskip

To verify, in a first step, that the local boundedness property {\rm(B)} is satisfied in the given setting, we have to impose some growth conditions on the source function $f:\Omega \times \mathbb{R} \to \mathbb{R}$ from~\eqref{eq:f}. In the following, consider $r\ge 1$ such that the embedding $\Hone\hookrightarrow\mathrm{L}^{r+1}(\Omega)$ is continuous, i.e., there exists a constant $C_r>0$ (depending on $r$ and on $\Omega)$ such that
\begin{equation}\label{eq:Sobolev}
\Lpnorm{v}{r+1}\le C_r\Lnorm{\nabla v}\qquad\forall v\in\Hone.
\end{equation}
Indeed, by the Sobolev embedding theorem this holds true if
\begin{equation*}
\begin{cases}
0\le r\le\infty& \text{for $d=1$},\\
0\le r<\infty& \text{for $d=2$},\\
0\le r\le 5     & \text{for $d=3$};    
\end{cases} 
\end{equation*}
see, e.g., \cite[Thm.~4.12]{AdamsFournier:03}. 

\begin{proposition} \label{prop:PropertyB}
If there exist positive constants $c_1$ and $c_2$ such that 
\begin{align} \label{eq:Bfgrowth}
|f(\cdot,t)| \leq c_1 |t|^r+c_2 \qquad \text{a.e. in } \Omega, \ \forall t \in \mathbb{R},
\end{align}
for some $r>0$ as in~\eqref{eq:Sobolev}, then property {\rm(B)} is satisfied for any norm $\norm{\blank}$ on $\X=\Hone$ that is equivalent to~$\|\blank\|_{\Hone}$. 
\end{proposition}

\begin{proof}
Without loss of generality, for the purpose of this proof, we may suppose that $\norm{\blank}=\|\blank\|_{\Hone}$. Then,
for any given $w \in \X$ and $\varrho>0$, choose an arbitrary $v \in \X$ with $\norm{v-w} \leq \varrho$. Then, for any $z \in \X$ with $\norm{z}=1$, employing the upper bound in~\eqref{eq:psi'}, applying H\"older's inequality, and making use of the growth condition~\eqref{eq:Bfgrowth}, we find that
\begin{align*}
|\dprod{\E'(v)}{z}| &\leq \alpha \left(M_{\psi} \norm{v} \norm{z}+ c_1 \int_\Omega |v|^r |z| \dx + c_2 \int_\Omega |z| \dx\right)\\
& \leq \alpha \left( M_{\psi} \norm{v}\norm{z} + c_1 \Lpnorm{v}{r+1}^r\Lpnorm{z}{r+1} + c_2 |\Omega|^{\nicefrac{r}{r+1}} \Lpnorm{z}{r+1} \right).
\end{align*}
Further employing the Sobolev bound~\eqref{eq:Sobolev} leads to 
\begin{align*}
|\dprod{\E'(v)}{z}| &\leq \alpha \left(M_{\psi} \norm{v}\norm{z}+c_1 C_{r}^{r+1} \norm{v}^r \norm{z}+c_2 C_{r} |\Omega|^{\nicefrac{r}{r+1}} \norm{z} \right) \\
&\leq \alpha \left( M_\psi \norm{v} + c_1 C_{r}^{r+1} \norm{v}^r+c_2 C_{r} |\Omega|^{\nicefrac{r}{r+1}}\right),
\end{align*}
since $\norm{z}=1$. Finally, recalling that $\norm{v-w} \leq \varrho$, we conclude that
\begin{align*}
\dnorm{\E'(v)} \leq \alpha \left( M_\psi (\norm{w}+\varrho)+c_1 C_{r}^{r+1}(\norm{w}+\varrho)^r+c_2 C_{r} |\Omega|^{\nicefrac{r}{r+1}} \right),
\end{align*}
which concludes the proof.
\end{proof}

\subsection{Semilinear elliptic models}

We first consider a class of semilinear elliptic boundary value problems with the identity function $\Psi(t)=t$ and $\alpha=1$ in~\eqref{eq:diffre}, which is associated to the weak form of the Laplace operator in~\eqref{eq:E'}, and a possibly nonlinear reaction term~$f$, i.e. 
\begin{equation}\label{eq:poisson}
\begin{aligned}
\Delta u(\x) +f(\x,u(\x))&=0 \quad &&\x\in \Omega,\\
u(\x)&=0 \quad &&\x\in \partial \Omega. 
\end{aligned}
\end{equation}
 The weak formulation of~\eqref{eq:poisson} is to find $u\in\Hone$ such that
\begin{equation*}
\int_\Omega\nabla u\cdot\nabla v\dx=\int_\Omega f(\cdot,u)v\dx\qquad\forall v\in\Hone.
\end{equation*}
For a parameter $\eta \ge 0$, we introduce the \emph{linearization} operator $\A_\eta:\X \to \X'$ by
\begin{align}\label{eq:semilinearAstrong}
\A_{\eta}u:=-\Delta u +\eta u, \qquad u \in \X.
\end{align}
Based on this choice of $\A_\eta$, the iteration scheme~\eqref{eq:iteration} reads as
\begin{align*}
\int_\Omega \nabla (\un[n+1]-\un[n]) \cdot \nabla v \dx+\eta \int_\Omega (\un[n+1]-\un[n]) v \dx = -\int_\Omega \nabla u^{n} \cdot \nabla v \dx + \int_\Omega f(\x,u^n)v \dx,
\end{align*}
for all $v \in \X$ and $n \geq 0$. By rearranging the terms, the scheme transforms into
\begin{equation*}
\int_\Omega \nabla \un[n+1] \cdot \nabla v \dx+\eta \int_\Omega (u^{n+1}-u^n)v \dx=\int_\Omega f(\cdot,u^n)v \dx \qquad \forall v \in \X,\quad n\ge 0,
\end{equation*}
and coincides with the iteration method proposed in~\cite{AHW:23};  see also the closely related $L$-scheme presented in~\cite{pop:04,Radu:16,radupop:15,slodicka:02}. Introducing the $\A_\eta$-induced norm 
\begin{equation}\label{eq:etanorm}
\norm{v}^2:=\dprod{\A_\eta v}{v}=\Lnorm{\nabla v}^2+\eta\Lnorm{v}^2,\qquad v\in\Hone,
\end{equation}
and recalling Remark~\ref{rem:elliptic}, we note the following elementary observation.

\begin{lemma}[Properties (A1)--(A3)]\label{lem:A12'}
For any $\eta\ge 0$, the operator~$\A_\eta$ from~\eqref{eq:semilinearAstrong} satisfies the properties {\rm(A1)}--{\rm(A3)} as well as the conditions~\eqref{eq:Steinhaus} and~\eqref{eq:SteinhausD}.
\end{lemma}

\subsubsection{Reaction term with linear growth} \label{sec:lineargrowth}

We will verify the properties (T) and (E) first in the context of reactions of linear growth. To this end, we recall the structural assumptions on the nonlinearity~$f$ from~\cite{AHW:23}:

\begin{enumerate}[(i)]
\item $f(\cdot,s) \in \Ltwo$ for all $s \in \mathbb{R}$;
\item $f$ is differentiable in the second variable;
\item There exists a constant $\CC>0$ such that the set
\begin{equation}\label{eq:set}
\G(\CC):=\left\{\lambda>0:\,\sigma_f(\lambda)< \CC+\nicefrac{1}{\lambda}\right\}
\end{equation}
is non-empty, where we let
\begin{equation}\label{eq:sigma}
\sigma_f(\lambda):=\esssup_{\x \in \Omega} \sup_{u \in \mathbb{R}} \left|\frac{\partial f}{\partial u}(\x,u)+\frac{1}{\lambda}\right|,\qquad\lambda>0.
\end{equation}
\end{enumerate}
In particular, these hypotheses imply that the nonlinear reaction term $f$ features asymptotic linear growth in the second argument (as $|u|\to\infty$), \emph{without requiring monotonicity or convexity}. Moreover, as in~\cite{AHW:23}, we define the quantity
\begin{equation}\label{eq:us}
\mu_f:=\begin{cases}
2\left(\sup\G(\CC)\right)^{-1}&\text{if }\sup\G(\CC)<\infty,\\
0&\text{otherwise}.
\end{cases}
\end{equation}

To formulate the next result, we note the Poincar\'e inequality
\begin{equation}\label{eq:CP}
\Lnorm{v}^2\le C_P\Lnorm{\nabla v}^2\qquad\forall v\in\Hone.
\end{equation}

\begin{lemma}[Property~(T)]\label{lem:T'}
If the nonlinear reaction term $f$ satisfies the above conditions {\rm(i)--(iii)} for some $\rho>0$, and if
\begin{align} \label{eq:Deltat}
\eta > \frac{1}{2} \max\{0,\mu_f+\rho-C_P^{-1}\}=:\kappa_{f} \ge 0
\end{align}
in~\eqref{eq:semilinearAstrong}, with $\mu_f$ and $C_P$ from~\eqref{eq:us} and \eqref{eq:CP}, respectively, then property~{\em(T)} holds true.
\end{lemma}

\begin{proof}
Choose any $w \in \X$, and let $\delta_w:=\T(w)-w$. Then, referring to~\cite[Prop.~3.4]{AHW:23}, we notice the bound
\[
\E(w)-\E(\T(w))\ge \gamma_f\eta^{-1}\left(\Lnorm{\nabla\delta_w}^2+\eta\Lnorm{\delta_w}^2\right),
\]
with
\[
\gamma_f:=\eta\min\left\{\frac{\eta-\kappa_f}{\eta+C^{-1}_P},\frac{1}{2}\right\},
\]
where $\kappa_f$ is defined as in~\eqref{eq:Deltat}. Furthermore, from $\eta>\kappa_f$ we deduce that~$\gamma_f>0$. Therefore, recalling~\eqref{eq:etanorm} yields~\eqref{eq:Tpropb} with $\gamma:=\nicefrac{\gamma_f}{\eta}>0$.
\end{proof}

\begin{lemma}[Property~(E)]\label{lem:E}
If the nonlinear function $f$ satisfies {\rm(i)--(iii)} above for some $\rho>0$, then property~{\rm(E)} is satisfied.
\end{lemma}

\begin{proof}
We first note the compact Sobolev embedding $\Hone \comp \Ltwo$ in dimensions $1\le d\le 3$. Moreover, employing~\eqref{eq:sigma} for some~$\lambda\in\Lambda_f(\rho)$, cf.~\eqref{eq:set}, we obtain the uniform Lipschitz bound
\[
|f(\x,s)-f(\x,t)|
\le|f(\x,s)+\lambda^{-1}s-(f(\x,t)+\lambda^{-1}t)|+\lambda^{-1}|s-t|
 \leq (\sigma_f(\lambda)+\lambda^{-1}) |s-t|,
\]
for almost every $\x\in\Omega$ and all $s,t\in\R$. Therefore, for any sequence $\{v^n\}_n\subset\Hone$ that converges in $\Ltwo$ to a limit~$v^\star\in\Hone$, as presumed in the weak continuity property~\eqref{eq:G}, the Cauchy--Schwarz inequality yields
\begin{align*}
\left|\int_\Omega (f(\cdot,v^{n})-f(\cdot,v^\star))w\dx\right|
&\le (\sigma_f(\lambda)+\lambda^{-1}) \int_\Omega|v^{n}-v^\star| |w| \dx\\
&\le (\sigma_f(\lambda)+\lambda^{-1}) \Lnorm{v^{n}-v^\star}\Lnorm{w}\xrightarrow{n\to\infty}0,
\end{align*}
for each $w\in\X$. Consequently, Lemma~\ref{lem:compE} this yields the claim.
\end{proof}

Note that we have now verified all the assumptions on the operator $\A=\A_\eta$ and on the energy functional $\E$ in Theorem~\ref{thm:convergence}. Additional properties can be proved in the special case that $\rho>0$ in (iii) can be chosen such that $\rho\le C^{-1}_P$, with $C_P$ being the Poincar\'e constant from~\eqref{eq:CP}, see the following Lemma~\ref{lem:PS'}. In this context, we recall that the weak coercivity together with the weak sequential lower semicontinuity of $\E$ imply the boundedness of the (energy) sequence generated by our iteration scheme; cf.~Remarks~\ref{rem:boundedseq} and~\ref{rem:boundedenergie}.

\begin{lemma}\label{lem:PS'}
Let the assumptions {\rm(i)--(iii)} about the nonlinearity $f$ be fulfilled, and suppose that there exists $\rho\in(0,C^{-1}_P]$ with $\Lambda_f(\rho)\neq\emptyset$. Then, the following statements on the functional $\E$ from~\eqref{eq:diffre} hold true:
\begin{enumerate}[{\rm(a)}]
\item $\E$ is weakly coercive on $\Hone$.
\item $\E$ is weakly sequentially lower semicontinuous on $\Hone$.
\end{enumerate}
\end{lemma}

\begin{proof}
The weak coercivity of $\E$ follows from~\cite[Rem.~2.3]{AHW:23}. Moreover, recalling the arguments presented in~\cite[Rem.~2.5]{AHW:23}, we deduce that $\E$ is weakly sequentially lower semicontinuous.
\end{proof}

\subsubsection{Monotone potential models} \label{sec:powergrowth}

We assume that $f$ in~\eqref{eq:poisson} takes the form
\begin{equation}\label{eq:fxu}
f(\x,u):=h(\x)-\varphi(\mat x,u),
\end{equation}
with $h\in\mathrm{L}^{\nicefrac{(r+1)}{r}}(\Omega)$, where we assume that $r\ge 1$ is appropriately bounded so that the embedding $\Hone\hookrightarrow\mathrm{L}^{r+1}(\Omega)$ is continuous; cf.~\eqref{eq:Sobolev}. Moreover, we assume that the nonlinearity $\varphi$ satisfies the following properties:
\begin{enumerate}[({F}1)]
\item For any $y,z\in\R$ it holds the monotonicity property
\[
(\varphi(\cdot,y)-\varphi(\cdot,z))(y-z)\ge 0\qquad\text{a.e. in }\Omega.
\]

\item There are constants $c_1,c_2,r\ge0$ such that the growth bound
\begin{equation*}
|\varphi(\cdot,z)|\le c_1|z|^r+c_2\qquad\forall z\in\R\quad\text{a.e. in }\Omega
\end{equation*}
holds true.

\item There exists a monotone increasing function $\Lphi:\,[0,\infty)\to[0,\infty)$ such that, for any $u,v\in\Hone$, we have the estimate
\[
\int_\Omega|\varphi(\cdot,u)-\varphi(\cdot,v)||u-v|\dx\le  
\Lphi(\rho(u,v))\Lpnorm{u-v}{r+1}^2,
\]
where we let 
\begin{equation}\label{eq:rho}
\rho(u,v)=\max\big\{\Lpnorm{u}{r+1},\Lpnorm{u-v}{r+1}\big\}.
\end{equation}
\end{enumerate}\medskip

For a parameter $\alpha\in(0,1)$, we define the energy 
\begin{equation}\label{eq:Ealpha}
\E(v)
=\frac{\alpha}{2}\int_\Omega |\nabla v|^2\dx
+\alpha\int_\Omega \Phi(\cdot,v)\dx-\alpha\int_\Omega hv\dx,\qquad v\in\Hone;
\end{equation}
i.e., we let $\Psi(t)=t$ in~\eqref{eq:diffre}, and
\[
\F(\cdot,v)=\int_\Omega hv\dx-\int_\Omega\Phi(\cdot,v)\dx,
\]
whence we observe the G\^ateaux derivatives
\[
\dprod{\partial_v\F(\cdot,v)}{w}=\int_\Omega\left(h-\varphi(\cdot,v)\right)w\dx
=\int_\Omega f(\cdot,v)w\dx,\qquad v,w\in\Hone,
\]
and
\begin{equation}\label{eq:E'alpha}
\dprod{\E'(v)}{w}
=\alpha \int_\Omega\nabla v\cdot\nabla w\dx
-\alpha\int_\Omega f(\cdot,v)w\dx,\qquad v,w\in\Hone,
\end{equation}
cf.~\eqref{eq:f} and~\eqref{eq:E'}, respectively.

We will use the linearization operator $\A=-\Delta$, i.e., with $\eta=0$ in~\eqref{eq:semilinearAstrong}, so that the iteration~\eqref{eq:iteration} reads
\begin{equation}\label{eq:updatealpha}
\int_\Omega \nabla\un[n+1]\cdot\nabla v\dx
=(1-\alpha)\int_\Omega \nabla\un[n]\cdot\nabla v\dx
+ \alpha\int_\Omega f(\cdot,\un)v\dx \qquad\forall v\in\Hone,\quad n\ge 0,
\end{equation}
which is an iteration scheme of Zarantonello type; see, e.g.~\cite{HeidWihler:19v2}.

\begin{remark}
For any $v \in \X$, we note that $\E'(v)$ in~\eqref{eq:E'alpha} belongs to $\X'$, and property (B) is satisfied. Indeed, involving the growth condition~(F2), this follows by analogous arguments as in the proof of Proposition~\ref{prop:PropertyB}. 
\end{remark}

In light of Proposition~\ref{prop:strongconvergence} on the full and strong convergence of the sequence from Theorem~\ref{thm:convergence}, the following result is instrumental.

\begin{proposition}
[Strong monotonicity and property (E)]\label{prop:ME}
If the nonlinearity $f$ fulfills {\rm(F1)}, then $\E'$ from~\eqref{eq:E'alpha} is strongly monotone, viz. 
\[
\dprod{\E'(u)-\E'(v)}{u-v}\ge\alpha\|\nabla(u-v)\|^2_{\Ltwo},\qquad \forall u,v\in\Hone. 
\]
Moreover, $\E$ satisfies property {\rm (E)}, and is weakly sequentially lower semicontinuous and weakly coercive.
\end{proposition}

\begin{proof}
    For any $u,v\in\Hone$, recalling the representation of $\E'$ from~\eqref{eq:E'alpha}, and using (F1), we have
\begin{align}
\begin{split}
\dprod{\E'(u)-\E'(v)}{u-v}
&=\alpha\int_\Omega |\nabla(u-v)|^2\dx+\alpha\int_\Omega(\varphi(\cdot,u)-\varphi(\cdot,v))(u-v)\dx\label{eq:3.6}\\
&\ge\alpha\|\nabla(u-v)\|^2_{\Ltwo}.
\end{split}
\end{align}
This shows that $\E'$ is strongly monotone, and the remaining claims follow immediately from Lemma~\ref{lem:monotone}.
\end{proof}

It remains to establish property~(T). For this purpose, for $w\in\Hone$, we notice that the update operator $\T$ from~\eqref{eq:T} for the energy functional~\eqref{eq:Ealpha} and the linearization operator $\A=-\Delta$ is given through the weak formulation
\begin{equation}\label{eq:Talpha}
\int_\Omega \nabla\T(w)\cdot\nabla v\dx
=(1-\alpha)\int_\Omega \nabla w\cdot\nabla v\dx
+ \alpha\int_\Omega f(\cdot,w)v\dx\qquad\forall v\in\Hone,
\end{equation}
cf.~\eqref{eq:updatealpha}.

\begin{lemma}[Property~(T)]\label{lem:MT}
    Let $\un[0]\in\Hone$, and suppose that {\rm (F1), (F2) and (F3)} are satisfied. Then, for $\alpha\in(0,1)$ in~\eqref{eq:Ealpha} sufficiently small, the operator $\T$ from~\eqref{eq:Talpha} satisfies property~{\rm(T)}.
\end{lemma}

\begin{proof}
We proceed in several steps.
\begin{enumerate}[(i)]

\item \emph{Local Lipschitz continuity:} For $v,w\in\Hone$, recalling~\eqref{eq:3.6} and making use of~(F3), we have that
\begin{align*}
\alpha^{-1}\left|\dprod{\E'(v)-\E'(w)}{v-w}\right|
&\le\Lnorm{\nabla(v-w)}^2
+\int_\Omega|\varphi(\cdot,v)-\varphi(\cdot,w)||v-w|\dx\\
&\le\Lnorm{\nabla(v-w)}^2+\Lphi(\rho(v,w))\Lpnorm{v-w}{r+1}^2,
\end{align*}
with $\rho$ from~\eqref{eq:rho}. Hence, employing the Sobolev embedding~\eqref{eq:Sobolev}, results in
\begin{align*}
\left|\dprod{\E'(v)-\E'(w)}{v-w}\right|
&\le\alpha\left(1+C^2_r\Lphi(\rho(v,w))\right)\Lnorm{\nabla(v-w)}^2.
\end{align*}

\item \emph{Boundedness of the set $\M(\un[0])$:}
This follows directly from the strong monotonicity of $\E'$, cf.~Proposition~\ref{prop:ME}, and from Remark~\ref{rem:Mbd}.

\item \emph{Boundedness of the update:} 
For an arbitrary $w\in\Hone$, we define the update $\delta_w:=\T(w)-w$. Then, from~\eqref{eq:Talpha}, with $v=\delta_w$, we infer the identity
\[
\alpha^{-1}\|\nabla\delta_w\|^2_{\Ltwo}
=-\int_\Omega\nabla w\cdot\nabla\delta_w\dx+\int_\Omega f(\cdot,w)\delta_w \dx.
\]
Hence, by involving~(F2), which allows to proceed along the lines of the proof of Proposition~\ref{prop:PropertyB}, we find that
\begin{equation}\label{eq:dbound}
    \|\nabla\delta_w\|_{\Ltwo}\le \alpha C_r\left(c_1C^r_r\Lnorm{\nabla w}^r+c_2|\Omega|^{\nicefrac{r}{(r+1)}}+\Lpnorm{h}{\nicefrac{(r+1)}{r}}\right)+\alpha\Lnorm{\nabla w}.
\end{equation}

\item \emph{Property~{\rm(T)}:}
In view of Remark~\ref{rem:Tequiv}, for any $w\in\M(\un[0])$ and $\delta_w:=\T(w)-w$, it suffices to show that there exists $\gamma >0 $ such that
\begin{equation}\label{eq:Tequiv1}
\int_0^1\dprod{\E'(s\delta_w+w)-\E'(w)}{\delta_w}\,\d s
\le (1-\gamma)\Lnorm{\nabla\delta_w}^2.
\end{equation}
Indeed, applying the local Lipschitz continuity~(i), we notice that
\[
\int_0^1\dprod{\E'(s\delta_w+w)-\E'(w)}{\delta_w}\,\d s
\le \alpha\Lnorm{\nabla\delta_w}^2\int_0^1\left(1+C^2_r\Lphi(\rho(w,w+s\delta_w)\right)s\,\d s.
\]
Using that $0\le s\le 1$, and applying~\eqref{eq:Sobolev} yields
\[
\Lpnorm{s\delta_w}{r+1}
\le \Lpnorm{\delta_w}{r+1}
\le C_r\Lnorm{\nabla\delta_w}.
\]
Owing to~\eqref{eq:dbound}, and exploiting the boundedness of $w$ from~(ii), it follows that there exists a bounded constant $\rho_0$ (independent of $w$) such that
\[
\rho(w,w+s\delta_w)\le\rho_0\qquad\forall s\in[0,1].
\]
Thus, due to the monotonicity of $\Lphi$ in (F3), we conclude that
\begin{align*}
\int_0^1\dprod{\E'(s\delta_w+w)-\E'(w)}{\delta_w}\,\d s
\le\frac{\alpha}{2}\left(1+C^2_r\Lphi(\rho_0)\right)\Lnorm{\nabla\delta_w}^2.
\end{align*}
Therefore, choosing $\alpha<2\left(1+C^2_r\Lphi(\rho_0)\right)^{-1}$ shows the desired estimate~\eqref{eq:Tequiv1} with 
\[
\gamma=1-\nicefrac{\alpha\left(1+C^2_r\Lphi(\rho_0)\right)}{2}>0.
\]
\end{enumerate}
This completes the proof.
\end{proof}

\subsection{Quasilinear monotone models}

Finally, we will consider second-order elliptic quasilinear boundary value problems in divergence form. Specifically, we seek an element $u \in \X=\Hone$ such that
\begin{align} \label{eq:diffeq}
\int_\Omega \psi(|\nabla u|^2) \nabla u \cdot \nabla v \dx=\int_\Omega hv \dx \qquad \forall v \in \X,
\end{align}
where $\psi$ is a nonlinear diffusion coefficient that satisfies the monotonicity bounds~\eqref{eq:mu}, and $h \in \Ltwo$ is a given source function. For the problem~\eqref{eq:diffeq}, under suitable assumptions on $\psi$, there exists several iterative linearized solvers including the Zarantonello iteration, the Ka\v{c}anov scheme, and the Newton method; we refer to~\cite{HeidWihler:19v2} for a unified approach. In the current work, we will solely focus on a novel modified Ka\v{c}anov scheme introduced in~\cite{HeidWihler:22}, which allows for an analysis under relaxed assumptions on the diffusion coefficient. For that purpose, for $\alpha \in (0,\infty)$, we define the energy functional
\begin{align} \label{eq:Equasi}
\E(u):=  \frac{\alpha}{2} \int_\Omega \Psi(|\nabla u|^2) \dx - \alpha \int_\Omega h u \dx,
\end{align}
where we let
\[
\Psi(s)=\int_0^s \psi(t) \dt,
\]
and $\F(\cdot,u)=hu$ in~\eqref{eq:diffre}.
Then, we consider the resulting iterative procedure given by 
\begin{equation}\label{eq:Kacanov}
\int_{\Omega} \psi(|\nabla u^n|^2) \nabla (u^{n+1}-u^n) \cdot \nabla v \dx = - \alpha \int_\Omega \psi(|\nabla u^n|^2) \nabla u^n \cdot \nabla v \dx + \alpha \int_\Omega hv \dx
\end{equation}
for all $v \in \X$. In particular, in light of~\eqref{eq:iteration}, for given~$u\in\X$, the  linearization operator $\A[u]:\X \to \X'$ is defined through
\[
\dprod{\A[u]v}{w}=\int_\Omega \psi(|\nabla u|^2) \nabla v \cdot \nabla w \dx, \qquad v,w \in \X.
\]
From~\eqref{eq:psi'} we directly observe that $\A[\cdot]$ is uniformly bounded and coercive, and thus, satisfies the properties (A1)--(A3) as well as the conditions~\eqref{eq:Steinhaus} and~\eqref{eq:SteinhausD}; cf.~Remark~\ref{rem:elliptic}. In addition, using~\eqref{eq:mu}, we deduce that the derivative $\E'$ from~\eqref{eq:E'}, given by
\[
\dprod{\E'(u)}{v}=\alpha\int_\Omega \psi(|\nabla u|^2)\nabla u\cdot\nabla v\dx
-\alpha\int_\Omega hv\dx,\qquad u,v\in\Hone,
\]
is strongly monotone with $\eta=\alpha m_\psi$ in~\eqref{eq:strongmon}, viz.
\begin{align*}
\dprod{\E'(v)-\E'(w)}{v-w}
&=\alpha\int_\Omega \left(\psi(|\nabla v|^2)\nabla v-\psi(|\nabla w|^2)\nabla w\right)\cdot\nabla (v-w)\dx\\
&\ge \alpha m_\psi\Lnorm{\nabla(v-w)}^2,
\end{align*}
for all $v,w\in\X$. In turn, by Lemma~\ref{lem:monotone}, the energy functional $\E$ satisfies property (E), and is weakly coercive.\medskip

It remains to verify property~(T).

\begin{lemma}[Property (T)] \label{lem:Tdiffusion}
Let $\vartheta_\psi:=\nicefrac{3M_\psi}{2m_\psi}>0$, with the constants $m_\psi, M_\psi$ from~\eqref{eq:mu}. If $\alpha\in(0,\vartheta^{-1}_\psi)$ in~\eqref{eq:Equasi}, then property {\rm (T)} is satisfied with $\gamma=1-\alpha\vartheta_\psi>0$.
\end{lemma}

\begin{proof}
For any $w\in\X=\Hone$ and $\delta_w=\T(w)-w$, applying the upper bound in~\eqref{eq:mu} as in~\cite[p.~550]{Zeidler:90}, it holds that
\begin{multline*}
\int_0^1\dprod{\E'(s\delta_w+w)-\E'(w)}{\delta_w}\,\d s\\
=\alpha\int_\Omega\left(\psi(|\nabla (s\delta_w+w)|^2)\nabla(s\delta_w+w)
-\psi(|\nabla w|^2)\nabla w\right)\cdot\nabla\delta_w\dx\ds
\le \frac32\alpha M_\psi\Lnorm{\nabla\delta_w}^2.
\end{multline*}
Furthermore, using the lower bound in~\eqref{eq:mu}, we observe that $m_\psi\Lnorm{\nabla\delta_w}^2\le\dprod{\A[w]\delta_w}{\delta_w}$, from which we deduce~\eqref{eq:Tequiv} with $\gamma>0$ as above.
\end{proof}

\begin{remark} \label{rem:classicalKacanov}
Under the stronger assumption that $\psi$, in addition to satisfying~\eqref{eq:mu}, is continuously differentiable and monotonically decreasing, the convergence of~\eqref{eq:Kacanov} is guaranteed for $\alpha =1$; see, e.g., \cite[Prop.~25.35]{Zeidler:90} for an analysis of the classical Ka\v{c}anov scheme.
\end{remark}

\section{Numerical experiments} \label{sec:NE}

In this section, we test the proposed energy-based adaptive Algorithm~\ref{alg:AILFEM} in combination with the finite element space enrichment strategy from Algorithm~\ref{alg:refen} for polynomial degree $p=1$ in the context of nonlinear diffusion-reaction models as discussed in \S\ref{sec:drmodels}. In all our experiments, we consider two-dimensional polygonal domains, with the Euclidean coordinates denoted by $\x=(x,y) \in \mathbb{R}^2$. Moreover, the computations are initiated on a uniform
and coarse triangulation of~$\Omega$, and the starting guess is chosen as $u_0^0 \equiv 0$. Furthermore, for simplicity, we choose $\gamma_N=1$ for all $N \geq 0$ in Algorithm~\ref{alg:AILFEM}, and the D\"orfler marking parameter in Algorithm~\ref{alg:refen} is set to $\theta=0.4$  for each of the experiments below. We note that a different choice of the adaptivity sequence $\{\gamma_N\}_N$ may lead to a potentially smaller or larger number of iterative linearization steps on the sequence of discrete spaces $\X_N$.

\subsection{Semilinear elliptic model with asymptotically linear reaction term} \label{exp:semilinear} We start with a problem that fits into the framework of \S\ref{sec:lineargrowth}. Specifically, for a coefficient $\nu > 0$, we consider the nonlinear reaction function
\[
f(u) = \nu \log(1+|u|)-u+1, \qquad u \in \X,
\]
in~\eqref{eq:poisson} posed on the unit square $\Omega:=(0,1) \times (0,1)$. Since $\nu > 0$, the problem at hand admits a non-negative solution, cf.~\cite[\S2.1]{JR:2022}. 
We will run this experiment for $\nu=1$, and set $\eta=2$ in~\eqref{eq:semilinearAstrong}, which is in line with our theory. Since the analytical solution and corresponding energy value are not available, we have computed a reference energy $\E_{\mathsf{ref}}$ by running our algorithm until the number of degrees of freedom exceeded $5 \cdot 10^5$. Subsequently, for the purpose of our plots, we have performed the experiment until the degrees of freedom were greater than $10^5$. In Figure~\ref{fig:energydecay} we display the energy difference 
\begin{equation}\label{eq:endiff}
\mathcal{E}_N:=\left|\E(u_N^{n^\star(N)})-\E_{\mathsf{ref}}\right|
\end{equation}
corresponding to the final approximation on each finite element space, and the number of linearization steps against the number of degrees of freedom; we see that the energy decays at a (presumably optimal) rate of $\mathcal{O}(\dim(\X_N)^{-1})$, and that the number of energy reduction steps on each discrete space stays at a constant value of~1.

\begin{figure}
{\includegraphics[width=0.8\textwidth]{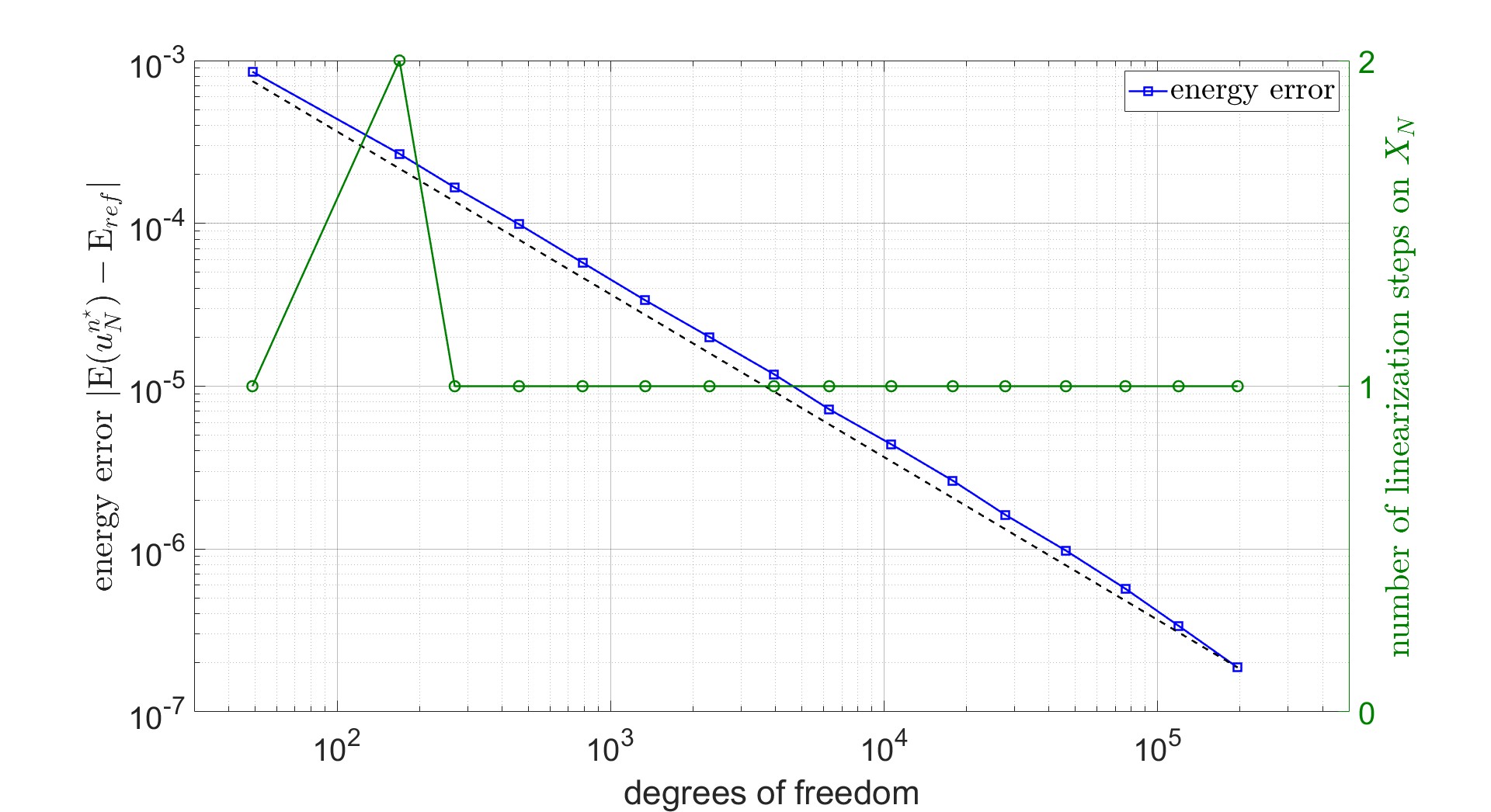}}
 \caption{Experiment~\ref{exp:semilinear}: The blue line displays the energy decay~$\mathcal{E}_N$ from~\eqref{eq:endiff} against the number of degrees of freedom (with the black dashed line indicating a rate of $-1$), and the green markers show the number of iterative linearization steps on the sequence of discrete spaces $\{\X_N\}_N$.}
 \label{fig:energydecay}
\end{figure}

\subsection{Power growth reaction} \label{exp:pg} Our second experiment concerns power-type growth models as outlined in \S\ref{sec:powergrowth}. In particular, the boundary value problem under consideration is the modified sine-Gordon equation from~\cite[Exp.~1]{brunner2024costoptimal}: On the unit square $\Omega := (0,1) \times (0,1)$, consider the equation 
\[
-\Delta u + u^3 + \sin(u) = h,
\]
with zero Dirichlet boundary conditions, where the source function $h(x,y)$ is chosen such that
\[
u^\star(x,y)=\sin(\pi x) \sin(\pi y) 
\]
solves the problem. We note that $\varphi(u)=u^3+\sin(u)$ in~\eqref{eq:fxu} with an associated potential $\Phi(u) = \tfrac{1}{4} u^4-\cos(u)+1$. It can be easily verified that the structural assumptions (F1)--(F3) are satisfied. Moreover, without extensive finetuning, we have observed that $\alpha=0.25$ is a reasonable choice for the damping parameter in the Zarantonello iteration~\eqref{eq:updatealpha}. In Figure~\ref{fig:PGplot}, we plot the error 
\begin{equation}\label{eq:eNn}
\mathrm{e}^\star_N:=\Lpnorm{\nabla\left(u^\star-u_N^{n^\star(N)}\right)}{2}
\end{equation}
against the number of degrees of freedom. We observe that the error decays at a linear rate, which is optimal for a $\mathbb{P}_1$-FEM. Moreover, for the given range, the number of iteration steps grows (at most) at a logarithmic rate. As was mentioned before, this could potentially be improved by choosing the sequence $\{\gamma_n\}_n$ in Algorithm~\ref{alg:AILFEM} more adequately.

\begin{figure}
{\includegraphics[width=0.8\textwidth]{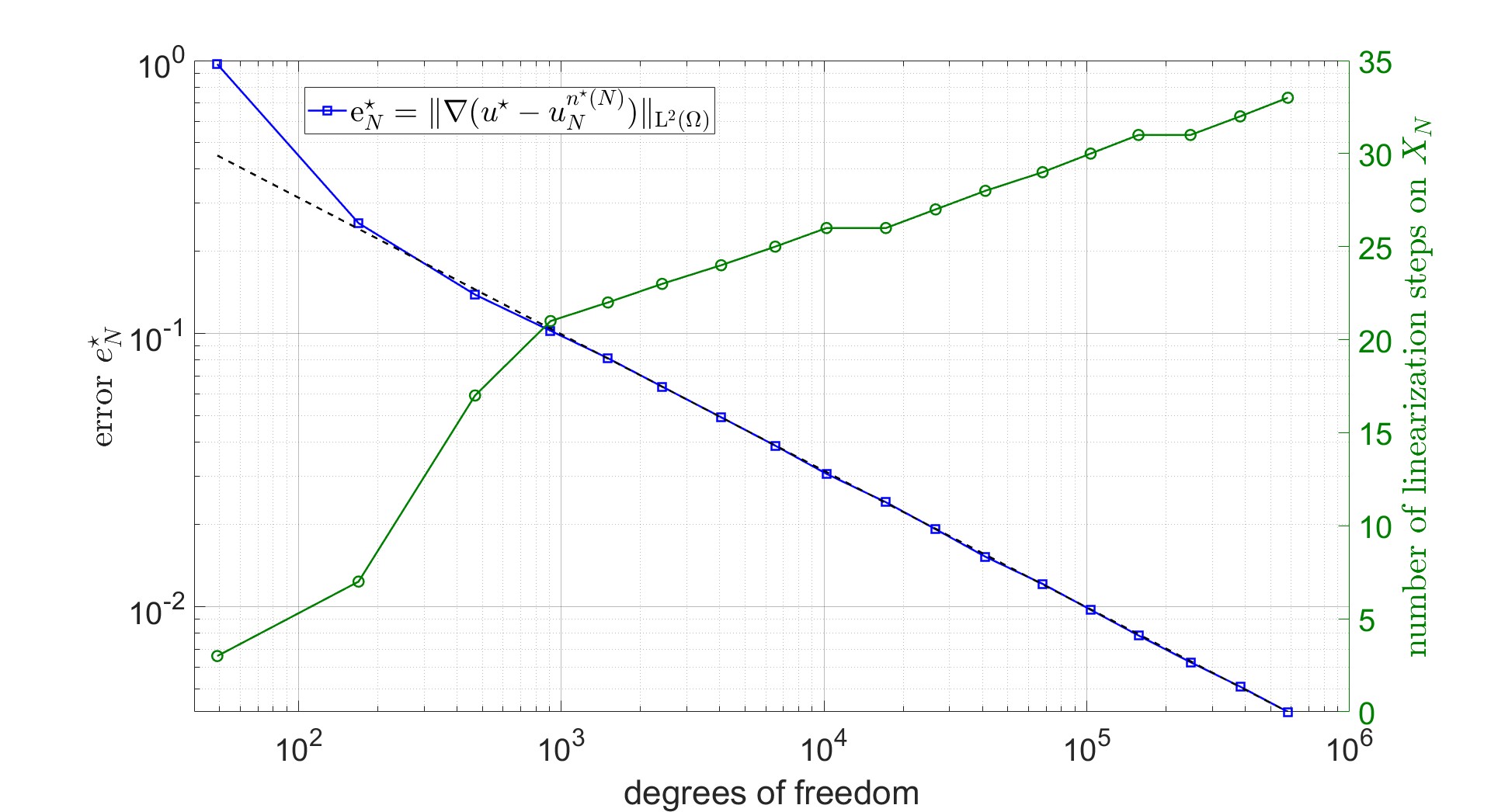}}
 \caption{Experiment~\ref{exp:pg}: The blue line displays the decay of the error $\mathrm{e}^\star_N$, cf.~\eqref{eq:eNn}, against the number of degrees of freedom (with the black dashed line indicating a presumably optimal rate of $-\nicefrac12$), and the green markers show the number of iterative linearization steps on the sequence of discrete spaces $\{\X_N\}_N$. }
 \label{fig:PGplot}
\end{figure}

\subsection{Stationary conservation law} \label{exp:scl}

Finally, we run an experiment for a quasilinear stationary conservation law. For that purpose, we consider the L-shaped domain $\Omega = (-1,1)^2 \setminus [0,1] \times [-1,0]$, and the diffusion coefficient $\psi(t)=1+\mathrm{e}^{-t}$ in~\eqref{eq:diffeq}. We note that this function obeys the bounds in~\eqref{eq:mu}, and is both differentiable and monotonically decreasing, thus we may set $\alpha = 1$ in the Ka\v{c}anov scheme~\eqref{eq:Kacanov}; cf.~Remark~\ref{rem:classicalKacanov}. In this test, we will choose the right-hand side $h$ such that the unique solution of the problem~\eqref{eq:diffeq} is given by 
\[
u^\star(r,\varphi)=r^{\nicefrac{2}{3}}
\sin (2 \varphi /3) (1-r \cos(\varphi))(1 + r \cos(\varphi))(1 - r \sin(\varphi))(1 + r \sin(\varphi)) \cos(\varphi),
\]
where $r$ and $\varphi$ are polar coordinates. We note that our solution exhibits a classical elliptic singularity at the re-entrant corner $(0,0)$. In Figure~\ref{fig:SCLmesh} we depict an intermediate adaptively refined mesh with $8 \, 576$ degrees of freedom. As we can observe, the mesh was mainly refined in a vicinity of the singularity; i.e. at the origin. In turn, we obtain an optimal decay rate of the error~$\mathrm{e}^\star_N$ from~\eqref{eq:eNn} with respect to the number of degrees of freedom; see Figure~\ref{fig:SCLplot}. In light of the singularity of the solution, this indeed highlights that our energy-based mesh-refinement strategy is highly effective. Similarly as before, within the given range, the number of iteration steps grows at most logarithmically, but it might level off. Once more, this observation might occur due to a suboptimal choice of the sequence $\{\gamma_n\}_n$.

\begin{figure}
{\includegraphics[width=\textwidth]{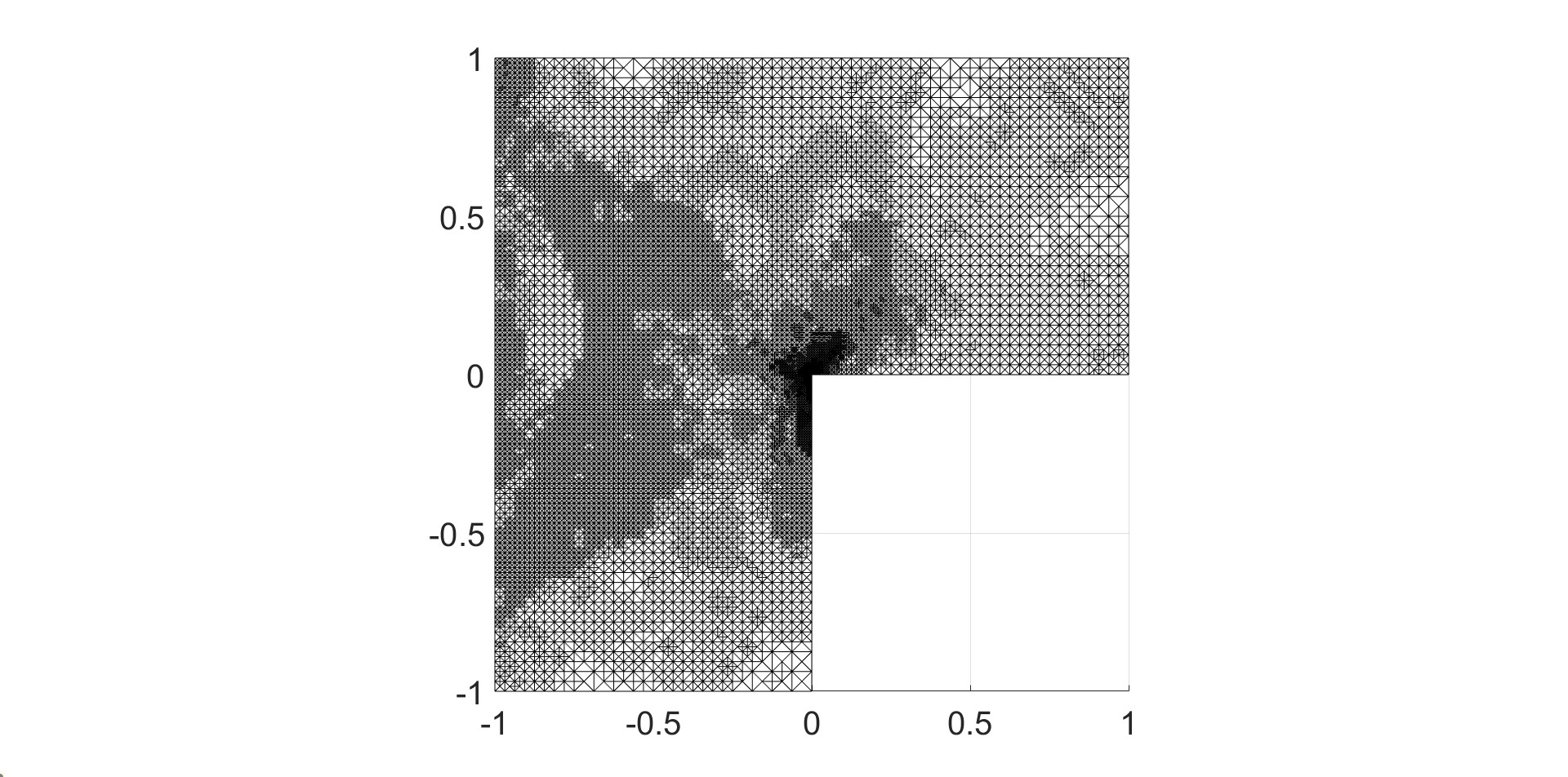}}
 \caption{Experiment~\ref{exp:scl}: Adaptively refined mesh with $8 \, 576$ degrees of freedom.}\label{fig:SCLmesh}
\end{figure}

\begin{figure}
{\includegraphics[width=0.8\textwidth]{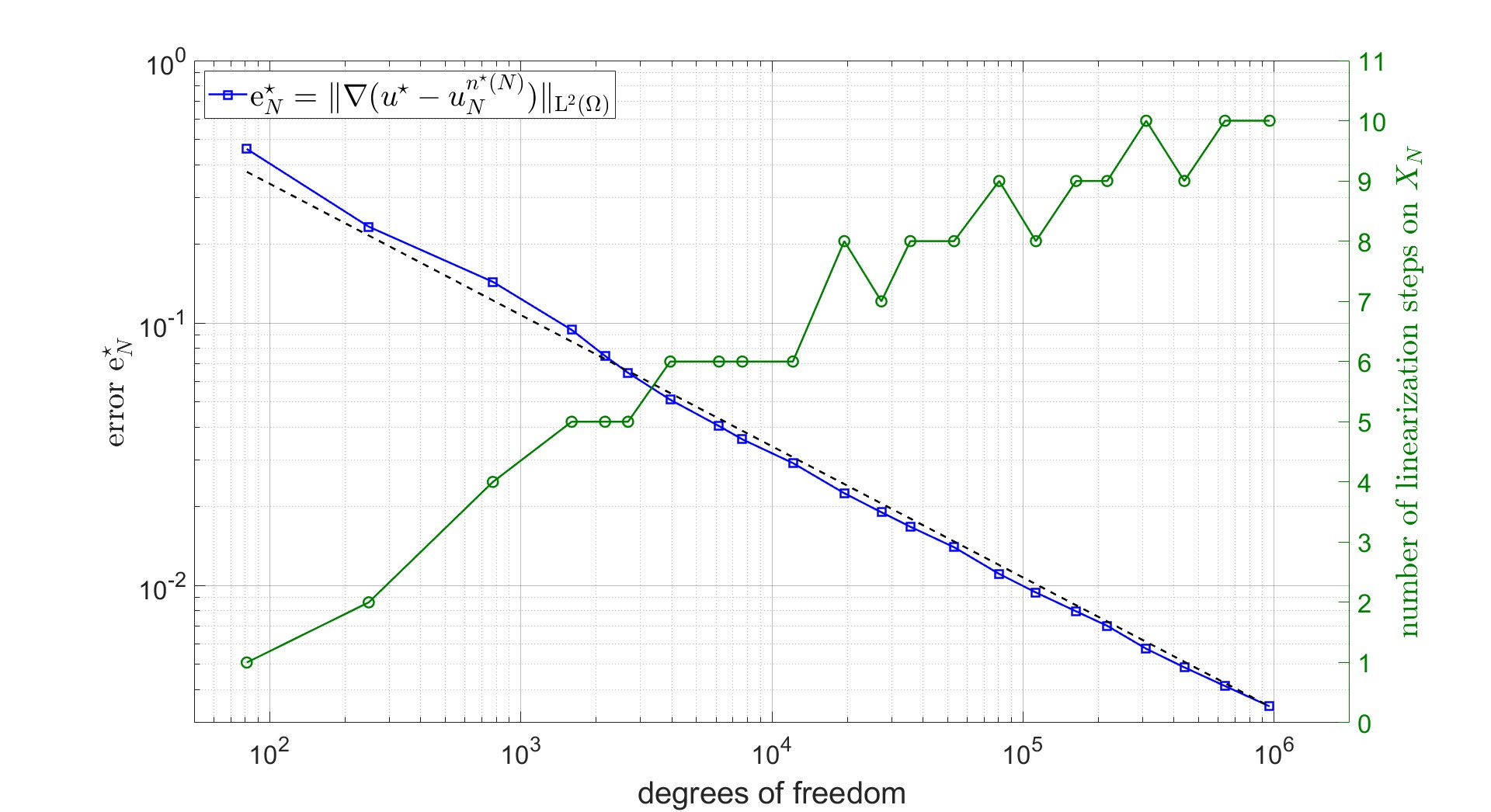}}
\caption{Experiment~\ref{exp:scl}: The blue line displays the decay of the error $\mathrm{e}^\star_N$, cf.~\eqref{eq:eNn}, against the number of degrees of freedom (with the black dashed line indicating a presumably optimal rate of $-\nicefrac12$), and the green markers show the number of iterative linearization steps on the sequence of discrete spaces $\{\X_N\}_N$.}\label{fig:SCLplot}
\end{figure}

\section{Conclusions}\label{sec:concl}

The central contribution of this work is the development of a general, abstract framework for the design and analysis of iterative numerical schemes that are explicitly guided by the energy structure of variational problems. Unlike more traditional approaches our methodology leverages the topology induced by the energy functional, and thereby leads to the formulation of practical algorithms---particularly relevant in the context of adaptive finite element methods---that remain effective even in the absence of classical error estimators. As such, the proposed framework serves to unify and generalize recent advances in energy-based adaptivity, and extends their applicability to a broader class of nonlinear problems.

\bibliographystyle{amsplain}
\bibliography{references}
\end{document}